%% file: main.tex
\documentclass[10pt,reqno]{amsart} 
\usepackage[pass]{geometry}
\newlength\DX
\paperwidth=\dimexpr\paperwidth-\DX\relax
\hoffset=\dimexpr\hoffset-.5\DX\relax
\newlength\DY
\paperheight=\dimexpr\paperheight-\DY\relax
\voffset=\dimexpr\voffset-.5\DY-.5\footskip\relax

\usepackage[utf8]{inputenc}
\usepackage[T1]{fontenc}    
\usepackage[colorlinks=true,citecolor=blue,linkcolor=blue]{hyperref}      
\usepackage{url}            
\usepackage{booktabs}       
\usepackage{amsfonts}       
\usepackage{amsthm}
\usepackage{amsmath}
\usepackage{amssymb}
\usepackage{mdframed}
\usepackage{nicefrac}       
\usepackage{microtype}      
\usepackage{xcolor}         
\usepackage{todonotes}
\usepackage{tikz}
\usepackage{tikz-cd}

\newcommand{\sca}[2]{\langle #1 | #2\rangle}
\newcommand{\nr}[1]{\left\Vert #1\right\Vert}
\newcommand{\abs}[1]{\left\vert #1\right\vert}
\newcommand{\Nsp}{\mathbb{N}}
\newcommand{\Rsp}{\mathbb{R}}
\newcommand{\NCov}{\mathcal{N}}
\newcommand{\id}{\mathrm{id}}
\newcommand{\DD}{\mathrm{D}}
\newcommand{\Class}{\mathcal{C}}
\newcommand{\Wass}{\mathrm{W}}

\newcommand{\dd}{\mathrm{d}}

\renewcommand{\L}{\mathrm{L}}

\newcommand{\Lip}{\mathrm{Lip}}
\newcommand{\diam}{\mathrm{diam}}

\newcommand{\eps}{\varepsilon}

\newcommand{\Prob}{\mathcal{P}}

\newcommand{\tr}{\mathrm{tr}}

\newcommand{\Esp}{\mathbb{E}}
\newcommand{\Kant}{\mathcal{K}}

\newcommand{\spt}{\mathrm{spt}}
\newcommand{\osc}{\mathrm{osc}}

\newcommand{\Haus}{\mathcal{H}}

\newcommand{\sign}{\mathrm{sign}}
\newcommand*{\Tan}{\mathcal{T}}
\newcommand{\Loss}{\mathcal{L}}

\newmdtheoremenv{theo}{Theorem}
\newmdtheoremenv{coro}{Corollary}

\newtheorem{theorem}{Theorem}[section]
\newtheorem*{theoremst}{Theorem}
\newtheorem{corollary}[theorem]{Corollary}
\newtheorem{lemma}[theorem]{Lemma}

\newtheorem{problem}[theorem]{Problem}

\newtheorem*{propositionst}{Proposition}

\theoremstyle{definition}

\newtheorem{remark}{Remark}[section]
\newtheorem{example}[theorem]{Example}

\title[Stability of the Pushforward by an Optimal Transport Map]{Quantitative Stability of the Pushforward Operation by an Optimal Transport Map}

\author{Guillaume Carlier}
\address{Ceremade, Univ. Paris-Dauphine PSL, 75775 Paris \and Mokaplan, Inria Paris}
\email{carlier@ceremade.dauphine.fr}

\author{Alex Delalande}
\address{Lagrange Mathematics and Computing Research Center, 75007, Paris, France}
\email{delalande.alex@gmail.com}

\author{Quentin Mérigot}
\address{Université Paris-Saclay, CNRS, Laboratoire de mathématiques d’Orsay, 91405, Orsay, France \and Institut universitaire de France}
\email{quentin.merigot@universite-paris-saclay.fr}

\begin{document}

\begin{abstract}
   We study the quantitative stability of the mapping that to a measure associates its 
   pushforward measure by a fixed (non-smooth) optimal transport map. We exhibit a tight 
   Hölder-behavior for 
   this operation under minimal assumptions. 
   Our proof essentially relies on a new bound that quantifies the size of 
   the singular sets of a convex and Lipschitz continuous function on a bounded domain.
\end{abstract} 

\maketitle


\smallskip


\section{Introduction}
\label{sec:introduction}

\input{1_introduction}

\section{Covering number of near-singularity sets of convex functions}
\label{sec:singular-sets}
\input{2_singularity_convex_functions}

\section{Stability of the pushforward by an optimal transport map} 
\label{sec:stability-pushforwards}

\input{3_stability_pushforward}

\medskip

\subsection*{Acknowledgement} 
The authors acknowledge the support of the Lagrange Mathematics and Computing Research Center and of the ANR (MAGA, ANR-16-CE40-0014). 

\medskip

\bibliographystyle{plain}
\bibliography{ref}

\appendix
\section{Omitted proofs} 
\label{sec:appendix}
\input{9_appendix}

\end{document}

%% file: 1_introduction.tex
The optimal transport problem is a two-century old foundational 
optimization problem of optimal mass allocation in geometric domains 
\cite{Monge}. The theoretical 
study of this problem has allowed to define a natural geometry on spaces of 
probability measures that 
offers precious tools for tackling both theoretical and numerical questions 
 involving probability measures \cite{villani2008optimal,ambrosio2008gradient, santambrogio2015optimal, comp_OT}. 
 The main feature of this geometry 
is the \emph{Wasserstein distance}: on the set $\Prob_2(\Rsp^d)$ of 
probability measures with finite second moment over $\Rsp^d$, the 
($2$-)Wasserstein distance between two measures $\rho, \mu \in \Prob_2(\Rsp^d)$, 
denoted $\Wass_2(\rho, \mu)$, is 
defined as the 
square-root of the value of the following minimization problem:
\begin{equation}
    \label{eq:ot-problem}
    \min_{\gamma \in \Gamma(\rho, \mu)} \int_{\Rsp^d \times \Rsp^d} \nr{x-y}^2 \dd \gamma(x,y),
\end{equation}
where $\Gamma(\rho, \mu)$ denotes the set of \emph{transport plans} 
or \emph{couplings} between $\rho$ and $\mu$, that is the set of probability 
measures over $\Rsp^d \times \Rsp^d$ with first marginal $\rho$ and second 
marginal $\mu$. Endowed with the Wasserstein distance, the metric space 
$(\Prob_2(\Rsp^d), \Wass_2)$ is a geodesic space referred to as the 
\emph{Wasserstein space}. In this space,  the (constant-speed) geodesics 
connecting 
two measures $\rho$ 
and $\mu$ in $\Prob_2(\Rsp^d)$ are given by the paths 
$\left( ((1-t)p_1 + t p_2)_\# \gamma \right)_{t \in [0, 1]}$ for any $\gamma \in \Gamma(\rho, \mu)$ 
that minimizes \eqref{eq:ot-problem}, where $p_1 : (x,y) \mapsto x$ and 
$p_2 : (x, y) \mapsto y$ are the projections onto the first and second 
coordinates respectively and where $f_\# \nu$ denotes the image measure of a 
measure $\nu$ under a map $f$. 
Interestingly, the Wasserstein space has found a physically-relevant 
pseudo-Riemannian structure, which 
has been leveraged to describe some well known 
evolution PDEs (such as the Fokker-Planck or porous medium equations) 
as gradient flows of some energy functionals on the space of
probability distributions \cite{Otto1998, JKO, otto2001geometry,ambrosio2008gradient}. 
In this formal Riemannian interpretation (formal because $(\Prob_2(\Rsp^d), \Wass_2)$ 
is not locally homeomorphic to a Euclidean space or even a Hilbert space), 
the geometric tangent cone $\Tan_\rho \Prob_2(\Rsp^d)$ to 
$\Prob_2(\Rsp^d)$ at a 
measure $\rho \in \Prob_2(\Rsp^d)$ can be described as the closure 
of the set 
\begin{equation*}
    \{ (p_1, \lambda(p_2 - p_1))_\# \gamma \mid \lambda > 0, \gamma \in \Prob_2(\Rsp^d \times \Rsp^d), 
(p_1)_\# \gamma = \rho, \spt(\gamma) \subset \partial \phi, \text{ $\phi$ convex}\}
\end{equation*}
with respect to an appropriately chosen Riemannian metric (see Chapter 12 
of \cite{ambrosio2008gradient}). In this expression, $\spt(\gamma)$ denotes 
the support of $\gamma$ and $\partial \phi$ denotes the 
subdifferential of the (proper and continuous) convex function 
$\phi : \Rsp^d \to \Rsp \cup \{ \infty \}$, 
that is the set
$$ \partial \phi = \{ (x,y) \in \Rsp^d \times \Rsp^d \mid \phi(x) + \phi^*(y) = \sca{x}{y} \},$$
where $\phi^*(\cdot) = \sup_{z \in \Rsp^d} \sca{z}{\cdot} - \phi(z)$ 
corresponds to the 
convex conjugate or Legendre transform of $\phi$. 

\subsection{Problem statement.} From above, it appears that the 
\emph{directions} of the elements of the tangent cone 
$\Tan_\rho \Prob_2(\Rsp^d)$ to 
$\Prob_2(\Rsp^d)$ at a probability measure $\rho$ are prescribed with 
convex functions. 
In the spirit of building a Riemannian logarithmic map, being given a 
new measure 
$\mu \in \Prob_2(\Rsp^d)$, one may wonder what are the 
possible \emph{directions} $\phi$ of the 
elements of 
$\Tan_\rho \Prob_2(\Rsp^d)$ that support the 
Wasserstein geodesics connecting $\rho$ to $\mu$. These 
can be recovered from the convex functions $\phi$ or $\psi^*$ that solve the 
following Kantorovich dual 
problems, which essentially correspond to the convex dual problems of 
\eqref{eq:ot-problem} 
(see e.g. Particular Case 5.16 in \cite{villani2008optimal}):
\begin{equation}
    \label{eq:kanto-dual-w2}
    \min_{\phi : \Rsp^d \to \Rsp \cup \{ \infty \}} 
    \int_{\Rsp^d} \phi \dd \rho + \int_{\Rsp^d} \phi^* \dd \mu = 
    \min_{\psi : \Rsp^d \to \Rsp \cup \{ \infty \}} 
    \int_{\Rsp^d} \psi^* \dd \rho + \int_{\Rsp^d} \psi \dd \mu.
\end{equation}
Conversely, in the spirit of building a Riemannian exponential map, 
being given a \emph{direction} from a convex function 
$\phi : \Rsp^d \to \Rsp \cup \{ \infty \}$, one may wonder what are 
the possible $t=1$ 
endpoints of the geodesics starting from $\rho$ and with initial 
velocities \emph{directed} by $\phi$, that is of the form 
$(p_1, p_2 - p_1)_\# \gamma$ for a coupling $\gamma$ with 
first marginal equal to $\rho$ and with support 
included in $\partial \phi$. Here, the answer is simple and the corresponding 
endpoint is the measure $(p_2)_\# \gamma$. Note that whenever $\phi$ is 
differentiable $\rho$-almost-everywhere, there is only one coupling $\gamma$ 
with first marginal equal to $\rho$ and support in $\partial \phi$: 
it is given by the 
coupling $\gamma = (\id, \nabla \phi)_\# \rho$ and in this case $\nabla \phi$ 
corresponds to the optimal transport map from $\rho$ to 
$\mu = (\nabla \phi)_\# \rho$ in the 
Brenier sense \cite{Brenier89}. 
In this setting, the \emph{exponential map} 
from the base point $\rho$ applied to the direction $\phi$ reduces 
to the pushforward measure $(\nabla \phi)_\# \rho$.

In this article, we are 
concerned with the quantitative stability with respect to the base point 
of the above-described \emph{exponential mapping} (or pushforward operation) 
in a fixed direction. Namely, we investigate the following 
problem:

\begin{problem}
    \label{pb:statement}
    Let $\phi : \Rsp^d \to \Rsp \cup \{ \infty\}$ be a 
    \emph{fixed}, proper and continuous convex function. Let
    $\rho, \tilde{\rho} \in \Prob_2(\Rsp^d)$ and consider $\gamma, \tilde{\gamma} 
    \in \Prob_2(\Rsp^d \times \Rsp^d)$ that are such that 
    $(p_1)_\# \gamma = \rho$,
    $(p_1)_\# \tilde{\gamma} = \tilde{\rho}$, 
    $\spt(\gamma) \subset \partial \phi$ and 
    $\spt(\tilde{\gamma}) \subset \partial \phi$. Under what conditions on 
    $\phi, \rho, \tilde{\rho}$ and how can one upper 
    bound $\Wass_2((p_2)_\# \gamma, (p_2)_\# \tilde{\gamma})$ in terms of 
    $\Wass_2(\rho, \tilde{\rho})$?
\end{problem}    

As mentioned above, whenever the function $\phi$ is differentiable $\rho$- 
and $\tilde{\rho}$-almost-everywhere in Problem~\ref{pb:statement}, the 
question it raises becomes that of upper bounding 
$\Wass_2((\nabla \phi)_\# \rho, (\nabla \phi)_\# \tilde{\rho})$ in terms of 
$\Wass_2(\rho, \tilde{\rho})$, which is the question of the quantitative 
stability of the pushforward operation by an optimal transport map.

\subsection{Motivations.} While of 
a theoretical nature, Problem~\ref{pb:statement} finds its relevance 
in several applied contexts. In order to motivate our study, we mention some 
of these contexts in what follows.

\subsubsection{Numerical resolution of the Kantorovich dual.}
Many numerical methods that aim at solving the optimal transport problem 
\eqref{eq:ot-problem} between two measures $\rho$ and $\mu$ in $\Prob_2(\Rsp^d)$ 
rely on the dual problems exposed in \eqref{eq:kanto-dual-w2} (see 
\cite{comp_OT, MERIGOT2021133} for surveys on such methods). 
Focusing for instance on the right-hand side 
problem in \eqref{eq:kanto-dual-w2}, it is possible to add the 
constraint $\int_{\Rsp^d} \psi \dd \mu = 0$ in this problem without altering 
its value. 
The resolution of \eqref{eq:ot-problem} can thus be reduced to the minimization 
of the \emph{Kantorovich functional} 
$\Kant_\rho : \psi \mapsto \int_{\Rsp^d} \psi^* \dd \rho$ under the constraint 
$\int_{\Rsp^d} \psi \dd \mu = 0$. The functional $\Kant_\rho$ being convex, its 
minimization is amenable to first- and second-order optimization methods. 
In these methods, the user must be able to evaluate the \emph{gradient} of 
$\Kant_\rho$ at a given $\psi$. Formally, this gradient reads
$$ \nabla \Kant_\rho(\psi) = - (\nabla \psi^*)_\# \rho. $$
Whenever $\rho$ is absolutely continuous, the numerical computation of such a 
gradient can be challenging in dimension $d \geq 3$. 
This happens for instance in the setting of semi-discrete optimal transport (where 
in addition the target $\mu$ is assumed to be discrete) that is used 
to model Euler incompressible equations \cite{deGoes15, gallouet:hal-01425826}, in 
computational geometry \cite{levy15}, optics design \cite{meyron18} or in 
cosmology \cite{levy22}. In this case, the user might instead consider a 
finitely supported approximation 
$\tilde{\rho} = \frac{1}{N}\sum_i \delta_{x_i}$ of $\rho$, and set
$$ \tilde{\mu} := - \frac{1}{N} \sum_i \delta_{y_i} $$ 
as an approximation for $\nabla \Kant_\rho(\psi)$, where in this definition 
each $y_i$ is chosen as an element of the subdifferential $\partial \psi^*(x_i)$. 
This measure $\tilde{\mu}$ is very easy to compute in practice 
(one only needs to compute elements of the subdifferential of $\psi^*$). 
This procedure then raises the question of the quality of the approximation of 
$\nabla \Kant_\rho(\psi)$ offered by $\tilde{\mu}$ in terms 
of the quality of the approximation of $\rho$ given by $\tilde{\rho}$, which 
is an instance of Problem~\ref{pb:statement}. 


\subsubsection{Computation of geodesics and barycenters in the Linearized Optimal Transport framework.} 
In \cite{LOT_ref_image}, the above-described 
pseudo-Riemannian structure of the Wasserstein space was leveraged to 
\emph{linearize} the optimal transport geometry in order to perform tractable data 
analysis tasks on measure-like data, giving birth to the \emph{Linearized 
Optimal Transport} (LOT) framework. In this framework, 
an absolutely continuous reference 
measure $\rho \in \Prob_2(\Rsp^d)$ is chosen and fixed. Because $\rho$ is absolutely 
continuous, any continuous convex function $\phi : \Rsp^d \to \Rsp$ 
is differentiable $\rho$-almost everywhere, so that the tangent bundle 
$\Tan_\rho \Prob_2(\Rsp^d)$ to $\Prob_2(\Rsp^d)$ at $\rho$ can be regarded (see 
Chapter 8 of \cite{ambrosio2008gradient}) as the 
$\L^2(\rho; \Rsp^d)$ closure of 
$$ \{  \lambda (\nabla \phi - \id) \mid \lambda > 0, \phi \text{ convex} \}. $$
Then, the LOT framework maps any new measure 
$\mu \in \Prob_2(\Rsp^d)$ to $\L^2(\rho; \Rsp^d)$ via the embedding 
$\mu \mapsto \nabla \phi_\mu - \id \in \L^2(\rho; \Rsp^d)$ where $\phi_\mu$ is any 
minimizer of the left-hand side dual problem in \eqref{eq:kanto-dual-w2}. 
This can be seen as sending $\mu \in \Prob_2(\Rsp^d)$ 
into the (linear) tangent 
space $\Tan_\rho \Prob_2(\Rsp^d) \subset \L^2(\rho; \Rsp^d)$ via a 
Riemannian 
logarithmic map. The advantage of employing this embedding is to enable the use of 
all the Hilbertian tools of statistics and machine learning on datasets of probability 
measures, somehow consistently with the Wasserstein geometry. Note that working with 
this embedding is equivalent to replacing the Wasserstein distance with the distance
$$ \Wass_{2, \rho}(\mu, \nu) := \nr{\nabla \phi_\mu - \nabla \phi_\nu}_{\L^2(\rho; \Rsp^d)}.$$
This distance, with respect to which geodesics are called 
\emph{generalized geodesics} in \cite{ambrosio2008gradient}, has been shown 
to be Hölder-equivalent in some settings to the original 
Wasserstein distance $\Wass_2$ in 
\cite{pmlr-v108-merigot20a, delalande:hal-03164147}, 
justifying to some extent the successes of the LOT framework witnessed 
on tasks of pattern recognition \cite{LOT_ref_image, LOT5, LOT4, LOT6}, 
generative modeling \cite{LOT3} or image processing \cite{LOT7}. A 
key advantage of the LOT embedding is that its image is convex, in the 
sense that any convex combination of embeddings  
provide a \emph{valid new embedding}. 
More precisely, for a dataset $(\mu_i)_{1 \leq i \leq N}$ of $N \geq 1$ 
probability measures in $\Prob_2(\Rsp^d)$ and a set $(\alpha_i)_{1 \leq i \leq N}$ of non-negative weights summing to one, the $\L^2(\rho, \Rsp^d)$-barycenter 
$\sum_{i=1}^N \alpha_i (\nabla \phi_{\mu_i} - \id)$ 
of the embeddings of each $\mu_i$ is a valid element of 
$\Tan_\rho \Prob_2(\Rsp^d)$ since it reads as the gradient of a convex 
function minus identity. One can thus apply the above-described exponential map 
to this barycenter of embeddings in order to define the measure 
\begin{equation*}
    \bar{\mu} := \left(\sum_{i=1}^N \alpha_i \nabla \phi_{\mu_i}  \right)_\# \rho.
\end{equation*}

\begin{figure}
    \centering
    \includegraphics[scale=0.5]{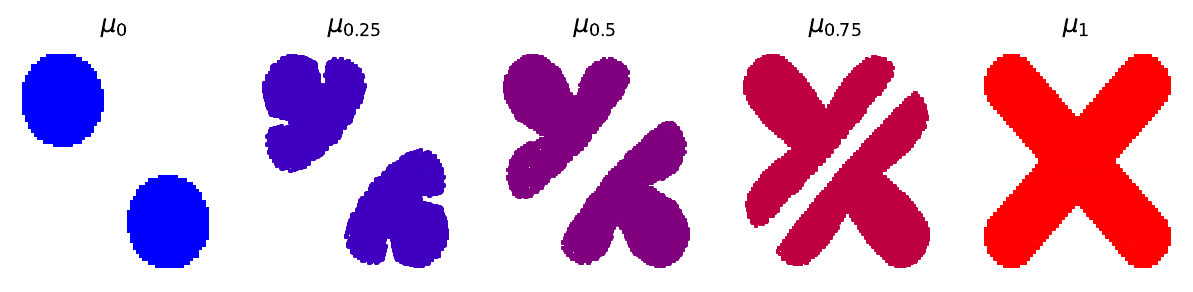}
    \caption{\emph{Linearized Optimal Transport} barycenters, or \emph{generalized geodesic}, between the discrete probability measures $\mu_0, \mu_1 \in \Prob([0,1]^2)$ (colored pixels indicate the position of support points). For $k \in \{0, 1\}$, the optimal transport map $\nabla \phi_k$ between the Lebesgue measure $\rho$ on $[0,1]^2$ and $\mu_k$ is computed using the Python package \texttt{pysdot} and \cite{Kitagawa2019ANA}. Then for $N=70^2$, define $\frac{1}{N}\sum_{i=1}^{N} \delta_{x_i}$ to be a discrete approximation of $\rho$ on a uniform grid. For $t \in (0, 1)$, the interpolant is defined as $\mu_t = \frac{1}{N} \sum_{i=1}^{N}\delta_{y_i^t}$, where each $y_i^t$ is chosen in $\partial ( (1-t) \phi_0 + t \phi_1)(x_i)$.}
    \label{fig:interpolation-pushforwards}
\end{figure}

\noindent The measure $\bar{\mu}$ gives a notion of average of the 
dataset $(\mu_i)_{1 \leq i \leq N}$ with respect to the weights $(\alpha_i)_{1 \leq i \leq N}$ (see Figure \ref{fig:interpolation-pushforwards} for an illustration in the case $N=2$ with varying weights). It may be used in place of the notion 
of Wasserstein barycenter \cite{agueh:hal-00637399}, which is defined as any 
minimizer of 
\begin{equation}
    \label{eq:w2-barycenter}
    \min_{\mu \in \Prob_2(\Rsp^d)} \sum_{i=1}^N \alpha_i
    \Wass_2^2(\mu, \mu_i).
\end{equation}
Wasserstein barycenters provide geometrically meaningful notions of averages of datasets 
of probability measures and have found many successful applications 
\cite{image_processing, geometry_processing, language_processing_1, language_processing_2, language_processing_3, JMLR:v19:17-084, pmlr-v32-cuturi14, pmlr-v70-ho17a}. 
However, the numerical resolution of \eqref{eq:w2-barycenter} is often tedious and 
working with the proxy $\bar{\mu}$ is often preferable since it essentially requires solving 
$N$ optimal transport problems between $\rho$ and each $\mu_i$. Nonetheless, it also 
requires computing the pushforward of $\rho$ by the map 
$\sum_{i=1}^N \alpha_i \nabla \phi_{\mu_i}$, which can be difficult in dimension 
$d \geq 3$. In practice, as for the computation of the gradient of the Kantorovich 
functional above, the user may approximate $\bar{\mu}$ by first 
discretizing $\rho$ and then pushing forward this discretization by the the map 
$\sum_{i=1}^N \alpha_i \nabla \phi_{\mu_i}$. The problem of 
controlling the bias induced by this process then gives another instance of 
Problem~\ref{pb:statement}. 

\subsubsection{Generative modeling with ICNNs.} 
Over the last decade, tools 
from optimal transport have made an increasing number of successful incursions in 
large-scale machine learning problems. These incursions are in part due to the introduction 
in \cite{pmlr-v70-amos17b} of the \emph{Input Convex Neural Networks} (ICNNs). These are 
neural networks whose architecture constraints them to be convex with respect 
to their input. Such networks were shown to be able to approximate arbitrarily well in supremum 
norm any 
convex Lipschitz function 
on a bounded domain~\cite{Chen18convex}. ICNNs have been used in the context of generative 
modeling through optimal transport, 
where one typically wants to learn a \emph{model} for a data probability distribution $\mu$ 
through the observation of samples from it. The optimal transport approach of this problem 
generally sees $\mu$ as the pushforward of a chosen simple probability distribution $\rho$ 
(typically a Gaussian) by an optimal transport map. This transport map must be 
learned: in \cite{taghvaei20192wasserstein, pmlr-v119-makkuva20a, korotin2021wasserstein, bunne22}, 
it is parametrized as the gradient of an ICNN 
$\phi_\theta$ parametrized by $\theta$, and the loss functions used in these works 
to find the right parameter are essentially proxies for the loss  
$$ \Loss(\theta) = \Wass_2( (\nabla \phi_\theta)_\# \rho, \mu). $$
In practice however, neither the source nor target distributions 
$\rho$ and $\mu$ are directly 
usable or known, and the user has to deal with statistical approximations 
$\hat{\rho}$ and 
$\hat{\mu}$ instead. This leads to the minimization of the empirical loss function
$$ \hat{\Loss}(\theta) = \Wass_2( (\nabla \phi_\theta)_\# \hat{\rho}, \hat{\mu}) $$
in place of the original loss function $\mathcal{L}$. In order to derive convergence rates 
for this empirical risk minimization problem, one may want to upper bound 
$|\hat{\Loss}(\theta) - \Loss(\theta)|$ in terms of $\Wass_2(\hat{\rho}, \rho)$ and 
$\Wass_2(\hat{\mu}, \mu)$. This reduces to yet another instance of 
Problem~\ref{pb:statement} after a use of the triangle inequality.

\subsection{Positive and negative results.} We expose here a positive result for 
Problem~\ref{pb:statement} in the case where $\phi$ is assumed to be regular. We also expose 
negative results in the case where no assumptions are made on $\phi, \rho$ and $\tilde{\rho}$, 
justifying this way the necessity for minimal assumptions.

\subsubsection{A positive result in the regular case} In the case 
where the convex function 
$\phi$ of Problem~\ref{pb:statement} is of class 
$\Class^{1, \alpha}$ for some $\alpha > 0$, the answer to the question raised in this problem 
is trivial:
\begin{propositionst}
    Let $\alpha \in (0,1]$ and $\phi \in \Class^{1, \alpha}(\Rsp^d)$ convex. 
    Then for any $\rho, \tilde{\rho} \in \Prob_2(\Rsp^d)$,
    $$ \Wass_2( (\nabla \phi)_\# \rho, (\nabla \phi)_\# \tilde{\rho}) \leq \nr{\nabla \phi}_{\Class^{0, \alpha}} \Wass_2(\rho, \tilde{\rho})^\alpha. $$
\end{propositionst}
\noindent This follows from Jensen's inequality and the fact that for any 
$\gamma \in \Gamma(\rho, \tilde{\rho})$, $(\nabla \phi, \nabla \phi)_\# \gamma$ 
is a valid coupling between $(\nabla \phi)_\# \rho$ and $(\nabla \phi)_\# \tilde{\rho}$. 
Even though this proposition brings an answer to Problem~\ref{pb:statement}, its outreach is 
limited. Indeed, when $\phi$ is an optimal transport potential (i.e. a solution to a dual 
problem of the type of~\eqref{eq:kanto-dual-w2}), getting regularity estimates for $\phi$ 
requires in general
to make strong regularity assumptions on the involved measures in order to be able to apply 
Caffarelli's regularity theory results \cite{Caffarelli1,Caffarelli2}, assumptions that are 
rarely satisfied in applications where at least one of the considered measures 
is often discrete. For instance, when $\phi$ is the dual solution of a semi-discrete 
optimal transport problem (with absolutely continuous source and discrete target), $\phi$ 
corresponds to a maximum of affine functions and as such it has many singularities. 
These singularities are actually often desirable, as for instance in the context 
of generative modeling of a data probability distribution $\mu$ with disconnected support 
as the pushforward of a Gaussian $\rho$ by the gradient of a convex function \cite{pmlr-v119-makkuva20a}.

\subsubsection{Negative results.} Whenever $\phi$ has singularities, it is very easy 
to build measures $\rho, \tilde{\rho}$ and couplings $\gamma, \tilde{\gamma}$ in 
Problem~\ref{pb:statement} that are such that it is not possible to control 
$\Wass_2( (p_2)_\# \gamma, (p_2)_\# \tilde{\gamma})$ in terms of $\Wass_2(\rho, \tilde{\rho})$. 
Consider for instance in dimension $d=1$ the case where $\phi = \abs{\cdot}$ is the 
absolute value. Then $\phi$ has a singularity at $0$:
$$ \partial \phi (0) = [-1, 1]. $$
A first negative result is available when both $\rho$ and $\tilde{\rho}$ are 
allowed to be discrete:
\begin{example}
    \label{ex:sources-discrete}
    Let $\phi = \abs{\cdot}$ on $\Rsp$. Let $\rho = \tilde{\rho} = \delta_0$ be the 
    the Dirac mass at zero and let $\gamma = \delta_{(0,1)}$ and 
    $\tilde{\gamma} = \delta_{(0,-1)}$. Then $(p_1)_\# \gamma = \rho$,
    $(p_1)_\# \tilde{\gamma} = \tilde{\rho}$, 
    $\spt(\gamma) \subset \partial \phi$ and 
    $\spt(\tilde{\gamma}) \subset \partial \phi$. However 
    $\Wass_2( (p_2)_\# \gamma, (p_2)_\# \tilde{\gamma}) = 2$ while 
    $\Wass_2(\rho, \tilde{\rho}) = 0$.
\end{example}

\noindent This example relies on placing both $\rho$ and $\tilde{\rho}$ at singularities 
of the convex function $\phi$. The set of singular points 
(i.e. points of non-differentiability) of a convex function defined on $\Rsp$ being 
at most countable, 
one can wonder what happens if we constraint one of the source measures in 
Problem~\ref{pb:statement} to be absolutely continuous with respect to the Lebesgue measure. 
Under such a constraint, it is still possible to build source measures that are arbitrarily 
close from each other but with pushforwards that are at a fixed non-zero distance 
from each other:

\begin{example}
    \label{ex:sources-discrete-and-ac-explodes}
    Let $\phi = \abs{\cdot}$ on $\Rsp$ and let $\eps>0$. Let $\rho = \delta_0$ and let 
    $\rho^\eps = \frac{1}{\eps} \lambda_{\vert [-\frac{\eps}{2}, \frac{\eps}{2}]}$ 
    be the rescaled Lebesgue measure restricted to $[-\frac{\eps}{2}, \frac{\eps}{2}]$. 
    Let $\gamma = \delta_{(0,1)}$ and $\gamma^\eps = (\id, \nabla \phi)_\# \rho^\eps$. 
    Then $(p_1)_\# \gamma = \rho$,
    $(p_1)_\# \gamma^\eps = \rho^\eps$, 
    $\spt(\gamma) \subset \partial \phi$ and 
    $\spt(\gamma^\eps) \subset \partial \phi$. However 
    $\Wass_2( (p_2)_\# \gamma, (p_2)_\# \gamma^\eps) = \sqrt{2}$ while 
    $\Wass_2(\rho, \rho^\eps) = \eps/2\sqrt{3}$.
\end{example}

\noindent Example~\ref{ex:sources-discrete-and-ac-explodes} relies on an absolutely continuous 
source 
measure $\rho^\eps$ whose density is allowed to explode so as to recover in the limit 
$\eps \to 0$ the pathological case of Example~\ref{ex:sources-discrete} with only discrete 
sources. In order to avoid this problem, we will make from now on the following 
minimal assumption in 
Problem~\ref{pb:statement}: one of the probability measures, say $\rho$, is absolutely 
continuous with respect to the Lebesgue measure and its density is upper bounded by 
some finite constant $M_\rho > 0$. Under this assumption, it is still possible to build an example 
(not as bad as Examples \ref{ex:sources-discrete} and 
\ref{ex:sources-discrete-and-ac-explodes}) showing that one cannot expect better than a 
Hölder-behavior for the pushforward operation:

\begin{example}
    \label{ex:holder-behavior} (See Figure~\ref{fig:holder-behavior} for an 
    illustration.)
    Let $\phi = \abs{\cdot}$ on $\Rsp$ and let $\eps \in (0, \frac{1}{2})$. Let 
    $\rho = \lambda_{[-\frac{1}{2}, \frac{1}{2}]}$ and let 
    $\rho^\eps = \lambda_{\vert[-\frac{1}{2}, -\frac{\eps}{2}] \cup [\frac{\eps}{2}, \frac{1}{2}]} + \eps \delta_0$. 
    Let $\gamma = (\id, \nabla \phi)_\# \rho$ and let 
    $\gamma^\eps = \int_{[-\frac{1}{2}, -\frac{\eps}{2}] \cup [\frac{\eps}{2}, \frac{1}{2}]} \delta_x \otimes \delta_{\nabla \phi (x)} \dd x + \eps \delta_{(0,1)}$. 
    Then $(p_1)_\# \gamma = \rho$,
    $(p_1)_\# \gamma^\eps = \rho^\eps$, 
    $\spt(\gamma) \subset \partial \phi$ and 
    $\spt(\gamma^\eps) \subset \partial \phi$. 
    Moreover, 
    $\Wass_2( (p_2)_\# \gamma, (p_2)_\# \gamma^\eps) = (2 \eps)^{1/2}$ while 
    $\Wass_2(\rho, \rho^\eps) = (\eps^3/12)^{1/2}$, so that 
    $\Wass_2( (p_2)_\# \gamma, (p_2)_\# \gamma^\eps) \sim \Wass_2(\rho, \rho^\eps)^{1/3}$.
\end{example}

\begin{figure}
    \centering
    \includegraphics[scale=0.55]{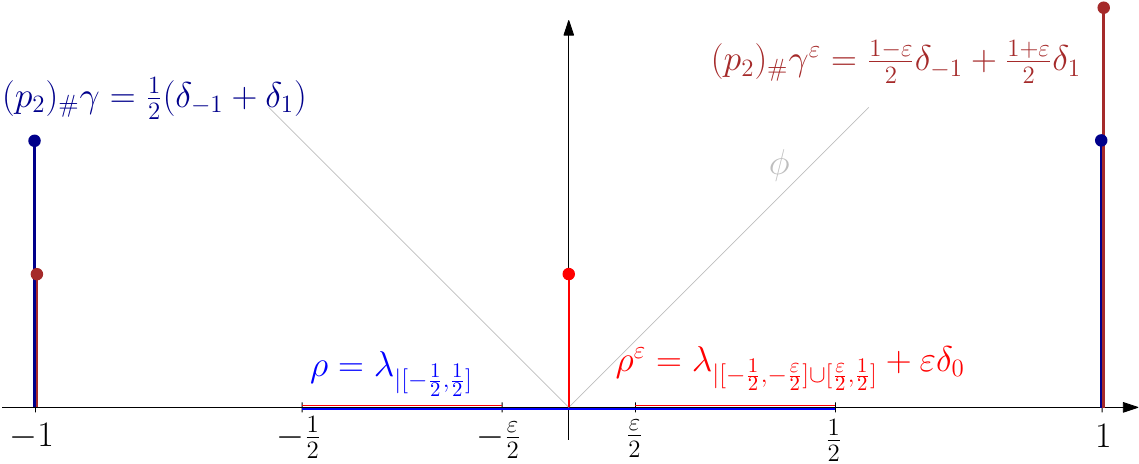}
    \caption{Illustration of Example~\ref{ex:holder-behavior}.}
    \label{fig:holder-behavior}
\end{figure}

\subsection{Contributions and outline.} In this article, we limit ourselves 
to a compact setting 
and work with probability measures supported in a ball $\Omega = B(0, R) \subset \Rsp^d$ 
centered at zero and of radius $R>0$. In this context, we assume that the convex function $\phi$ 
of Problem~\ref{pb:statement} is an $R$-Lipschitz continuous convex function in order to ensure 
that the 
\emph{pushforward} measures $(p_2)_\# \gamma$ and $(p_2)_\# \tilde{\gamma}$ of this problem 
also live in $\Omega$. We emphasize on the fact that we make \emph{no regularity assumption} 
on $\nabla \phi$. 

Our main result shows that, perhaps surprisingly, 
the situation described in Example~\ref{ex:holder-behavior} is as bad as it could get, and 
the Hölder-behavior being observed in this example is a general phenomenon:

\begin{theoremst}[Theorem~\ref{th:stability-pushforwards}]
Let $R>0$ and let $\Omega = B(0, R) \subset \Rsp^d$. Let $\phi : \Omega \to \Rsp$ be an 
$R$-Lipschitz continuous convex function. Let $\rho, \tilde{\rho} \in \Prob(\Omega)$ and 
assume that $\rho$ is absolutely continuous with density upped bounded by a constant 
$M_\rho \in (0, \infty)$. Then for any $\tilde{\gamma} \in \Prob(\Omega \times \Omega)$ 
such that $(p_1)_\# \tilde{\gamma} = \tilde{\rho}$ and $\spt(\tilde{\gamma}) \subset \partial \phi$,
$$ \Wass_2( (\nabla \phi)_\# \rho, (p_2)_\# \tilde{\gamma}) \lesssim \Wass_2(\rho, \tilde{\rho})^{1/3},$$
where $\lesssim$ hides an explicit multiplicative constant that depends on $d$, $R$ and $M_\rho$.
\end{theoremst}
We refer to Theorem~\ref{th:stability-pushforwards} in Section~\ref{sec:stability-pushforwards} 
for a more precise statement, with in particular an explicit expression for the hidden constant. 
Note also that the statement of Theorem~\ref{th:stability-pushforwards} is not limited to the 
context of quadratic optimal transport (i.e. optimal transport with respect to the cost 
$c(x,y)=\nr{x-y}^2$) but deals with the more general case of pushforwards by transport maps 
that are optimal with respect to the \emph{$p$-cost} $c(x,y)=\nr{x-y}^p$ for $p \geq 2$ ; and that 
the bounds are expressed in $\Wass_q$ and $\Wass_r$ distances (with $q$ and $r$ parameters 
to be chosen) in order to ensure the highest generality.

We are now in place of sketching the proof of our main result and the outline of the rest of 
the article. In Examples~\ref{ex:sources-discrete}, \ref{ex:sources-discrete-and-ac-explodes} and 
\ref{ex:holder-behavior}, we have seen that the \emph{instabilities} in the 
pushforward operation by the gradient of a convex Lipschitz function 
arise from the singularities of this function. 
Our main technical result, presented in Theorem~\ref{th:size-singular-sets-cvx-lip} of 
Section~\ref{sec:singular-sets} and which might be of independent interest, 
shows that on a bounded domain the number of singularities 
of such a function can be explicitly bounded. More precisely, for $\phi$ a convex Lipschitz 
function defined on $\Rsp^d$, we present in 
Theorem~\ref{th:size-singular-sets-cvx-lip} a tight upper bound on the covering numbers of the 
singular sets 
$$ \Sigma_{\eta, \alpha} = \{ x \in \Omega \mid \diam(\partial \phi(B(x,\eta))) \geq \alpha\}, $$
where $\alpha>0$ and $\eta>0$. In Remark~\ref{rk:alberti}, we note that the bound 
of Theorem~\ref{th:size-singular-sets-cvx-lip} may be seen as a refinement of 
a well-known result of 
Alberti, Ambrosio and Cannarsa \cite{Alberti1992}, who derived 
upper bounds on the dimension of the singular sets of semi-convex functions  
using measure-theoretic arguments, falling into the long line of works that studied the 
structure of the 
singularities of solutions to Hamilton-Jacobi equations 
\cite{Tsuji83, Tsuji86, Jensen87, Cannarsa87, Cannarsa89, Nakane91, Ambrosio1993} -- see 
\cite{Cannarsa2021} for a survey. As an immediate corollary to 
Theorem~\ref{th:size-singular-sets-cvx-lip} --~presented in 
Corollary~\ref{cor:estimate-grad-phi} of Section~\ref{sec:singular-sets}~-- 
we deduce that the function $\phi$ from this 
theorem satisfies the following integral estimate for any $\eta>0$:
\begin{equation}
    \label{eq:integral-bound-singularities}
    \int_{\Omega} \diam( \partial \phi(B(x, \eta)))^2 \dd x \lesssim \eta. 
\end{equation}
In Section~\ref{sec:stability-pushforwards}, after recalling some facts on the optimal 
transport problem with a general ground cost, we state and prove the main result 
Theorem~\ref{th:stability-pushforwards} that brings a tight answer to Problem~\ref{pb:statement}. 
This result essentially relies 
on~\eqref{eq:integral-bound-singularities} and a Markov bound. Let us sketch here the main idea: 
denote 
$S : \Omega \to \Omega$ the optimal transport map from $\rho$ to $\tilde{\rho}$ and for 
$\eta>0$, introduce the set $\Omega_\eta = \{ x \in \Omega \mid \nr{S(x) - x} \leq \eta \}$. Then, one has that 
$\Wass_2(\rho, \tilde{\rho}) = \nr{S - \id}_{\L^2(\rho, \Rsp^d)}$, so that 
Markov's inequality entails
\begin{equation}
    \label{eq:markov-bound-intro}
    \rho (\Omega \setminus \Omega_\eta) \lesssim \frac{\Wass_2^2(\rho, \tilde{\rho})}{\eta^2}. 
\end{equation}
Bounds \eqref{eq:integral-bound-singularities} and \eqref{eq:markov-bound-intro} allow 
to conclude: assuming here for simplicity that $\phi$ is differentiable 
$\tilde{\rho}$-almost-everywhere, we have for any $\eta>0$
\begin{align*}
    \Wass_2^2((\nabla \phi)_\# \rho, (\nabla \phi)_\# \tilde{\rho} ) 
    &\leq \int_{\Omega_\eta} \nr{\nabla \phi - \nabla \phi \circ S}^2 \dd \rho + \int_{\Omega \setminus \Omega_\eta} \nr{\nabla \phi - \nabla \phi \circ S}^2 \dd \rho \\
    &\leq  M_\rho \int_{\Omega_\eta} \diam( \partial \phi(B(x, \eta)))^2 \dd x + \int_{\Omega \setminus \Omega_\eta} (2R)^2 \dd \rho \\
    &\lesssim \eta + \frac{\Wass_2^2(\rho, \tilde{\rho})}{\eta^2}. 
\end{align*}
Setting $\eta = \Wass_2^{2/3}(\rho, \tilde{\rho})$ allows to reach the conclusion 
of Theorem~\ref{th:stability-pushforwards}.

%% file: 2_singularity_convex_functions.tex
The following result allows to quantify the size of the singular sets (i.e. points of non-differentiability) of a convex Lipschitz function on a bounded domain. 
We bound here the covering numbers of the sets of points $x$ in the domain for which there exist two \textit{nearby} points $x^\pm$, i.e. such that $\nr{x - x^{\pm}}\leq \eta$, where the gradients of the convex function $\phi$ are \textit{far} from each other, i.e. such that $\nr{\nabla \phi(x^+) - \nabla\phi(x^-)} \geq \alpha $. 
In this statement, $\NCov(K, \eta)$ denotes the minimum number of balls of radius $\eta>0$ that are needed to cover a compact set $K \subset \Rsp^d$.

\begin{theorem} 
\label{th:size-singular-sets-cvx-lip}
Let $\phi: \Rsp^d \to\Rsp$ be a convex and Lipschitz continuous function. Denote 
\begin{align*}
    \Sigma_{\eta,\alpha} &= \{ x \in \Rsp^d\mid \diam(\partial\phi(B(x,\eta)) \geq \alpha\}, \\
    \Sigma_\alpha &= \{ x\in \Rsp^d\mid \diam(\partial\phi(x)) \geq \alpha\}.
\end{align*} 
Then, for all $R\ > 0$, $\alpha$ and $\eta>0$, we have
$$ \NCov(\Sigma_{\eta,\alpha} \cap B(0,R), 8\eta) \leq c_{d,R,\eta} \frac{\Lip(\phi)}{\alpha \eta^{d-1}}, $$
with $c_{d,R,\eta} = 48 d^2 (R+4\eta)^{d-1}$. 
In particular, there exists a dimensional constant $c_d$ such that
$$ \Haus^{d-1}(\Sigma_\alpha\cap B(0,R))\leq c_d \frac{\Lip(\phi) R^{d-1}}{\alpha}. $$
\end{theorem}

As a corollary to this result, we get the following estimate that will prove useful in Section~\ref{sec:stability-pushforwards} for the study of the stability of the pushforward operation by an optimal transport map.
\begin{corollary} \label{cor:estimate-grad-phi} Let $\phi: \Rsp^d \to\Rsp$ be a convex and Lipschitz continuous function. Then for any $\eta, R > 0$ and $q>1$, 
\begin{equation*}
    \int_{B(0, R)} \diam(\partial \phi(B(x, \eta)))^q \dd x \leq c_{d, q, R, \eta} \Lip(\phi)^q \eta,
\end{equation*}
with $c_{d, q, R, \eta} = 48 d^2 \beta_d  2^{3d + q-1} \frac{q}{q-1}(R+4\eta)^{d-1}$, where $\beta_d$ denotes the volume of the unit ball of $\Rsp^d$.
\end{corollary}
\begin{proof}
    From Theorem~\ref{th:size-singular-sets-cvx-lip}, we directly get
    \begin{align*}
        \int_{B(0, R)} \diam(\partial \phi(B(x, \eta)))^q \dd x &= \int_0^{\infty} \abs{ \{ x \in B(0,R) \mid \diam( \partial \phi( B(x, \eta) ) )^q \geq t \} } \dd t \\
        &\leq \int_0^{(2 \Lip(\phi))^q} 48 d^2 (R+4\eta)^{d-1} \frac{\Lip(\phi)}{t^{1/q} \eta^{d-1}} \beta_d (8 \eta)^d \dd t \\
        &= c_{d, q, R, \eta} \Lip(\phi)^q \eta. \qedhere
    \end{align*}
\end{proof}

\begin{remark}[Singular sets of a convex Lipschitz function]
\label{rk:alberti}
    For $k \geq 1$, the $k$-singular set $\Sigma^k$ of $\phi : \Rsp^d \to \Rsp$ corresponds to the set of points $x$ in $\Rsp^d$ such that the Hausdorff dimension of $\partial\phi(x)$ is greater than or equal to $k$. The fact that the $k$-singular set of a convex Lipschitz function $\phi : \Rsp^d \to \Rsp$ is countably $\Haus^{d-k}$-rectifiable was established by Alberti, Ambrosio and Cannarsa in \cite{Alberti1992}. They also established the following estimate on the size of $\Sigma^k$:
    \begin{equation*}
        \int_{\Sigma^k \cap B(0, R)} \Haus^k(\partial \phi(x)) \dd \Haus^{d-k}(x) \leq c_d \left(\Lip(\phi) + 2R \right)^d,
    \end{equation*}
    where $c_d$ is a dimensional constant. With the notation of Theorem~\ref{th:size-singular-sets-cvx-lip}, taking $k=1$ in this estimate and using Markov's inequality allows to get the bound
    $$ \Haus^{d-1}(\Sigma_\alpha\cap B(0,R))\leq c_d \frac{\left(\Lip(\phi) + 2 R\right)^d}{\alpha}, $$
    that is similar in spirit to the bound we present in Theorem~\ref{th:size-singular-sets-cvx-lip}. However, the approach in \cite{Alberti1992} does not give an estimate on the covering numbers of $\Sigma_{\eta, \alpha}$, 
    which may prove necessary in specific contexts (see for instance the proof of 
    Theorem~\ref{th:stability-pushforwards} in the next section that relies on 
    Corollary~\ref{cor:estimate-grad-phi}). In this sense, the quantitative estimate of Theorem~\ref{th:size-singular-sets-cvx-lip} can be seen as a refinement of the estimate from \cite{Alberti1992} on the size of the set of non-differentiability points of a convex Lipschitz function on a bounded domain.
\end{remark}

\begin{remark}[Tightness]
    The bounds presented in Theorem~\ref{th:size-singular-sets-cvx-lip} are tight. Indeed, in dimension $d=1$, let $N \in \Nsp^*$ and $L, R > 0$ and define on $\Rsp$ the function
    $$ \xi :x \mapsto \max_{i=0, \dots, N} \left(\frac{2i}{N}-1\right)L x + \frac{2 L R}{N(N+1)}i(N-i).$$
    Then $\xi$ is convex and $L$-Lipschitz continuous (see Figure~\ref{fig:graph-xi} for an illustration of the graph of $\xi$ when $N=4$). Moreover, denoting $x_i = \left(\frac{2i}{N+1} - 1 \right)R$ for all $i$ in $\{1, \dots, N\}$, this function satisfies for all such $i$
    $$ \partial \xi (x_i) = \left[ \left(\frac{2i}{N}-1\right)L, \left(\frac{2(i+1)}{N}-1\right)L\right], $$
    and $\xi$ is differentiable everywhere else in $\Rsp \setminus \{x_i\}_{\{1 \leq i \leq N\}}$. In particular, setting $\alpha = \frac{2 L}{N}$, one can observe that for $\eta > 0$,
    \begin{align*}
    \Sigma_{\eta,\alpha} &= \{ x \in \Rsp \mid \diam(\partial\xi(B(x,\eta)) \geq \alpha\} = \bigcup_{1 \leq i \leq N} [x_i -\eta; x_i+ \eta], \\
    \text{and }\Sigma_\alpha &= \{ x\in \Rsp \mid \diam(\partial\xi(x)) \geq \alpha\} = \{ x_i \mid 1 \leq i \leq N\}, \\
    \end{align*}
    so that for $\eta \in (0, \frac{R}{8(N+1)})$, $$\Haus^{0}(\Sigma_\alpha\cap [-R, R]) = \NCov(\Sigma_{\eta,\alpha} \cap [-R, R], 8\eta) = N = \frac{2 L}{\alpha},$$
    which shows the tightness of the bounds of Theorem~\ref{th:size-singular-sets-cvx-lip} in dimension $d = 1$. In dimension $d \geq 1$, one may generalize this example by defining 
    $$ \phi :x \mapsto \xi(x^1),$$
    where $x^1$ is the projection of $x \in \Rsp^d$ on its first coordinate. Then, with the notations of Theorem~\ref{th:size-singular-sets-cvx-lip}, there are dimensional constants $c_d, \tilde{c_d}$ such that the convex and $L$-Lipschitz continuous function $\phi$ verifies for $\alpha = \frac{2 L}{N}$ and $\eta \in (0, \frac{R}{8(N+1)})$:
    \begin{align*}
    \NCov(\Sigma_{\eta,\alpha} \cap B(0,R), 8\eta) &\geq c_d \frac{R^{d-1}}{\eta^{d-1}} \frac{2 L}{\alpha} \\
    \Haus^{d-1}(\Sigma_\alpha\cap B(0,R)) &\geq \tilde{c_d} R^{d-1} \frac{2 L}{\alpha}. 
    \end{align*}
    These last inequalities also show the tightness of the bounds of Theorem~\ref{th:size-singular-sets-cvx-lip} with respect to $R$ and $\Lip(\phi)$. In comparison to the estimate obtained on the $(d-1)$-Hausdorff measure of $\Sigma_\alpha \cap B(0, R)$ deduced from \cite{Alberti1992} in Remark~\ref{rk:alberti}, this tightness corresponds to yet another refinement of the estimates from \cite{Alberti1992}.
\end{remark}

\begin{figure}
    \centering
    \includegraphics[scale=0.55]{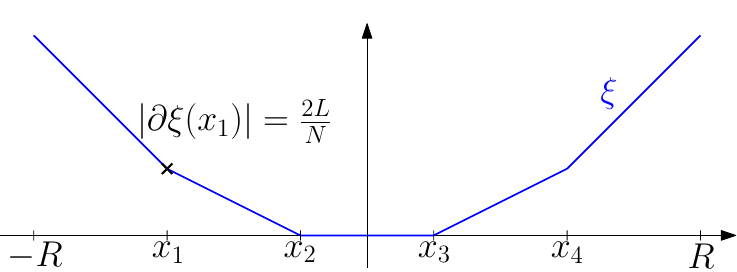}
    \caption{Graph of $\xi$ for $N=4$.}
    \label{fig:graph-xi}
\end{figure}

The proof of Theorem \ref{th:size-singular-sets-cvx-lip} uses the following lemma, similar to \cite[Lemma 3.2]{entropic-barycenters}, and whose proof is postponed after the proof of Theorem~\ref{th:size-singular-sets-cvx-lip}. 

\begin{lemma}
\label{lemma:bound-diam-grad-l1-norm-grad}
Let $\phi$ be a convex function over $\Rsp^d$. Then for any $x \in \Rsp^d$ and $\eta>0$,
$$ \diam( \partial \phi(B(x, \eta)) ) \leq \frac{12}{\beta_d \eta^d}  \nr{\nabla \phi}_{\L^1(B(x, 4\eta))}, $$
where $\beta_d$ denotes the volume of the unit ball of $\Rsp^d$.
\end{lemma}

With Lemma~\ref{lemma:bound-diam-grad-l1-norm-grad} in hand, we are now ready to prove Theorem~\ref{th:size-singular-sets-cvx-lip}.

\begin{proof}[Proof of Theorem~\ref{th:size-singular-sets-cvx-lip}] 
Let $\Sigma = \Sigma_{\eta,\alpha}$, and let $Z\subseteq \Sigma$ be a maximal $\eps$-packing of $\Sigma$ with $\eps=4\eta$, i.e. a finite subset of $\Sigma$ satisfying $\forall y\neq z\in Z$, $B(y,\eps)\cap B(z,\eps) =\emptyset$ and which is maximal with respect to the inclusion in the class of subsets of $\Sigma$ satisfying this assumption. We denote by $N$ the cardinal number of $Z$.
For any $x \in Z$, Lemma~\ref{lemma:bound-diam-grad-l1-norm-grad} gives us for any $c \in \Rsp^d$
\begin{equation}
\label{eq:before-PW2}
    \alpha \leq \diam \left(\partial \phi(B(x, \eta))\right) \leq \frac{12}{\beta_d \eta^d}   \nr{\nabla \phi - c}_{\L^1(B(x, 4\eta))}.
\end{equation}
Choosing $c = \frac{1}{\abs{B(x, 4\eta)}} \int_{B(x, 4\eta)} \nabla \phi(u) \dd u$, the Poincaré-Wirtinger inequality then ensures
$$ \nr{\nabla \phi - c}_{\L^1(B(x, 4\eta))} \leq 4\eta \int_{B(x, 4\eta)} \nr{\DD^2 \phi(u)}_{1,1} \dd u. $$
Using that for any positive semi-definite $d \times d$ matrix $M$, $\nr{M}_{1,1} \leq d \tr(M)$, we then have 
$$ \nr{\nabla \phi - c}_{\L^1(B(x, 4\eta))} \leq 4\eta d \int_{B(x, 4\eta)} \Delta \phi(u) \dd u, $$
where $\Delta$ stands for the Laplace operator. Injecting this last bound into \eqref{eq:before-PW2} yields
\begin{equation*}
    \alpha \leq \frac{48 d}{\beta_d} \frac{1}{\eta^{d-1}}  \int_{B(x, 4\eta)} \Delta \phi(u) \dd u,
\end{equation*}
Summing the last bound over $x \in Z$ and using that the balls of radius $4\eta\leq \eps$ centered at points of $Z$ do not intersect, we get
\begin{align*}
    \alpha N &\leq \frac{48 d}{\beta_d} \frac{1}{\eta^{d-1}} \sum_{x \in Z} \int_{B(x, 4 \eta)} \Delta \phi(u) \dd u \\
    &\leq \frac{48 d}{\beta_d} \frac{1}{\eta^{d-1}} \int_{B(0,R+4\eta)} \Delta \phi(u) \dd u \\
    &\leq \frac{48 d}{\beta_d} \omega_{d-1} (R+4\eta)^{d-1} \frac{\Lip(\phi)}{\eta^{d-1}}
\end{align*}
where we used an integration by part to get the last inequality and where $\omega_{d-1} = d \beta_d$ denotes the surface area of the $(d-1)$-unit sphere. Finally, we can easily check that $Z$ is a $2\eps$-covering of $\Sigma_{\alpha,\eta}$, implying the first bound of the statement.

To prove the second inequality, first note that $\Sigma_\alpha \subseteq \Sigma_{\alpha,\eta}$, so that for any 
$\eta \leq R$ one has
$$ \NCov(\Sigma_\alpha \cap B(0,R), \eta) \leq c_d \frac{\Lip(\phi) R^{d-1}}{\alpha \eta^{d-1}}. $$
We conclude using $\Haus^{d-1}(X) \leq c_d \liminf_{\eta\to 0} \eta^{d-1} \NCov(X,\eta)$, where $c_d$ is a dimensional constant.
\end{proof}

We finally prove Lemma~\ref{lemma:bound-diam-grad-l1-norm-grad}.
\begin{proof}[Proof of Lemma~\ref{lemma:bound-diam-grad-l1-norm-grad}]
Let $x \in \Rsp^d$ and $\eta>0$. One has by definition:
\begin{align*}
    \diam( \partial \phi(B(x, \eta)) ) &= \sup_{y, y' \in B(x, \eta)} \sup_{g \in \partial \phi(y), g' \in \partial \phi(y')} \nr{g - g'} \\
    &\leq \sup_{y, y' \in B(x, \eta)} \sup_{g \in \partial \phi(y), g' \in \partial \phi(y')} \nr{g} + \nr{g'} \\
    &= 2 \sup_{y \in B(x, \eta)} \sup_{g \in \partial \phi(y)} \nr{g} \\
    &= 2 \nr{\partial \phi}_{\L^\infty(B(x, \eta))}.
\end{align*}
But for any $y, y' \in \Rsp^d$ and $g \in \partial \phi(y)$, the convexity of $\phi$ entails
$$ \sca{g}{y' - y} \leq \abs{\phi(y') - \phi(y)}. $$
Therefore, choosing $y \in B(x, \eta)$ and $g \in \partial \phi(y)$ such that $\nr{\partial \phi}_{\L^\infty(B(x, \eta))} = \nr{g}$, one has for $y' = y + \eta \frac{g}{\nr{g}} \in B(y,\eta) \subset B(x, 2\eta)$ the following bound:
$$ \eta \nr{g} \leq \abs{\phi(y') - \phi(y)} \leq \osc_{B(x, 2\eta)}( \phi ), $$
where $\osc_K(f) = \sup_{u, v \in K} \abs{f(u) - f(v)}$. We thus have shown
\begin{equation}
    \label{eq:first-bound-diam-osc}
    \diam( \partial \phi(B(x, \eta)) ) \leq \frac{2}{\eta} \osc_{B(x, 2\eta)}( \phi ).
\end{equation}
We conclude exactly as in the proof of Lemma 3.2 of \cite{entropic-barycenters}, that we report here only for completeness: let $y_0 \in \arg\min_{B(x, 2\eta)} \phi, y_1 \in \arg\max_{B(x, 2\eta)} \phi, g_1 \in \partial \phi(y_1)$. Then by convexity of $\phi$, for any $y \in \Rsp^d$ and $g \in \partial \phi(y)$ one has
$$ \phi(y_1) + \sca{g_1}{y - y_1} \leq \phi(y) \leq \phi(y_0) + \sca{g}{y - y_0}. $$
It follows that
$$ \nr{g} \geq \frac{ \osc_{B(x, 2\eta)}(\phi) + \sca{g_1}{y - y_1} }{ \nr{y - y_0} }. $$
Introducing $W_\eta(y_1, g_1) = \{ y \in B(y_1, 2\eta) \vert \sca{g_1}{y - y_1} \geq 0\} \subset B(x, 4\eta)$, one then has
\begin{align*}
    \nr{\nabla \phi}_{\L^1(B(x, 4\eta))} &\geq \int_{W_\eta(y_1, g_1)} \nr{\nabla \phi} \dd y \\
    &\geq \int_{W_\eta(y_1, g_1)} \frac{ \osc_{B(x, 2\eta)}(\phi) }{ \nr{y - y_0} } \dd y \\
    &\geq \osc_{B(x, 2\eta)}(\phi) \int_{W_\eta(y_1, g_1)} \frac{1}{\nr{y - y_1} + \nr{y_1 - y_0}} \dd y \\
    &\geq \frac{ \osc_{B(x, 2\eta)}(\phi) }{6 \eta} \int_{B(y_1 + \eta \frac{g_1}{\nr{g_1}}, r)} \dd y \\
    &= \beta_d \frac{\eta^{d-1}}{6} \osc_{B(x, 2\eta)}(\phi),
\end{align*}
where $\beta_d$ denotes the volume of the unit ball of $\Rsp^d$ and where we used the fact that $B(y_1 + \eta \frac{g_1}{\nr{g_1}}, \eta) \subset W_\eta(y_1, g_1)$. Plugging this last bound into \eqref{eq:first-bound-diam-osc} finally yields
\begin{equation*}
    \diam( \partial \phi(B(x, \eta)) ) \leq \frac{12}{\omega_d \eta^d}  \nr{\nabla \phi}_{\L^1(B(x, 4\eta))}. \qedhere
\end{equation*}

\end{proof}

%% file: 3_stability_pushforward.tex
\subsection{Optimal transportation problem} We start this section by recalling some facts about 
the optimal transport problem with a general ground cost and discuss the existence and properties 
of optimal transport maps. 

\subsubsection{Primal and dual formulations.} Let $\Omega = B(0, R)$ be the open ball of $\Rsp^d$ centered at zero and of radius $R > 0$. For $\rho, \mu \in \Prob(\Omega)$ two probability measures supported over $\Omega$, Kantorovich's formulation of the optimal transport problem between $\rho$ and $\mu$ with respect to a continuous cost function $c : \Rsp^d \times \Rsp^d \to \Rsp$ corresponds to the following minimization problem:
\begin{equation}
\label{eq:ot-primal}
    \inf_{\Gamma(\rho, \mu)} \int_{\Omega \times \Omega} c(x,y) \dd \gamma(x,y).
\end{equation}
In this problem, the optimization is over the set $\Gamma(\rho, \mu)$ of couplings (or transport plans) between $\rho$ and $\mu$, i.e. the set of probability measures over $\Omega \times \Omega$ with first marginal $\rho$ and second marginal $\mu$. It is well-known (see e.g. Chapter 1 of \cite{santambrogio2015optimal}) that problem \eqref{eq:ot-primal} always admits a minimizer (possibly non-unique) and that it enjoys the following dual formulation, holding with strong-duality:
\begin{equation}
    \label{eq:ot-dual}
    \sup_{\varphi : \Omega \to \Rsp} \int_\Omega \varphi \dd \rho + \int_\Omega \varphi^c \dd \mu,
\end{equation}
where $\varphi^c(\cdot) = \inf_{x \in \Omega} c(x, \cdot) - \varphi(x)$ corresponds to the $c$-transform of $\varphi$. In turn, problem \eqref{eq:ot-dual} always admits a maximizer $\varphi$ (non-unique), which is referred to as a \emph{Kantorovich potential} and which must verify $(\varphi^c)^{\bar{c}} = \varphi$, where $\psi^{\bar{c}}(\cdot) = \inf_{y \in \Rsp^d} c(\cdot, y) - \psi(y)$ is a $\bar{c}$-transform.

\subsubsection{Wasserstein distances.} Whenever the cost function corresponds to the $p$-cost $c(x,y) = \xi_p(x-y)$ where $\xi_p(z) = \nr{z}^p$ for some $p \geq 1$, the $p$-th root of the value of problem \eqref{eq:ot-primal} defines the $p$-Wasserstein distance between the probability measures $\rho$ and $\mu$, denoted $\Wass_p(\rho, \mu)$. Wasserstein distances come with strong geometrical and physical interpretations that have made their success in many theoretical and applied contexts, see e.g. \cite{villani2008optimal, santambrogio2015optimal, comp_OT} for references.

When $p \geq 2$, the $p$-cost satisfies some immediate but strong regularity properties that we will exploit. In the following statement (whose proof can be found in the appendix), $f$ is said $\lambda$-concave with $\lambda \in \Rsp$ if $f + \frac{\lambda}{2}\nr{\cdot}^2$ is a concave function.

\begin{lemma}[Properties of $p$-cost] \label{lemma:pties-p-cost}
Let $p \geq 2$. On $\Omega = B(0, R)$, the mapping $z \mapsto \xi_p(z) = \nr{z}^p$ is strictly convex, of class $\Class^2$, $(p R^{p-1})$-Lipschitz continuous and  $(-p(p-1) R^{p-2})$-concave. The mapping $z \mapsto (\nabla \xi_p)^{-1}(z)$ is well-defined: for any $z \in \Rsp^d \setminus \{0\}$, 
\begin{gather*}
    (\nabla \xi_p)^{-1}(z) = \frac{1}{p^{\frac{1}{p-1}} \nr{z}^{\frac{p-2}{p-1}} } z,
\end{gather*}
and $ (\nabla \xi_p)^{-1}(0) = 0$. In particular, $(\nabla \xi_p)^{-1}$ is $\frac{1}{p-1}$-Hölder continuous:
\begin{equation*}
    \forall x, y \in \Omega, \quad  \nr{(\nabla \xi_p)^{-1}(y) - (\nabla \xi_p)^{-1}(x)} \leq \frac{3}{p^{\frac{1}{p-1}}} \nr{y - x}^{\frac{1}{p-1}}.
\end{equation*}
\end{lemma}


\subsubsection{Optimal transport maps.} By duality, one can observe that any $\gamma \in \Gamma(\rho, \mu)$ and $\varphi : \Omega \to \Rsp$ are respective solutions of problems \eqref{eq:ot-primal} and \eqref{eq:ot-dual} if and only if
\begin{equation}
\label{eq:primal-dual}
    \spt(\gamma) \subset \partial^c \varphi := \{ (x,y) \mid \varphi(x) + \varphi^c(y) = c(x,y)\},
\end{equation}
where $\spt(\gamma)$ denotes the support of $\gamma$. Incidentally, this observation allows to characterize cases of uniqueness of the solutions to problem \eqref{eq:ot-primal} depending on the choice of cost function $c$ and the assumptions made on the involved measures $\rho$ and $\mu$. Choose for instance the $p$-cost $c= \xi_p$ with $p \geq 2$ (see Section 1.3 of \cite{santambrogio2015optimal} for more general costs). Lemma~\ref{lemma:pties-p-cost} ensures that $\xi_p$ is Lipschitz continuous and  $\lambda$-concave with some explicit constants. These regularity properties are transmitted, with the same constants, to any Kantorovich potential $\varphi : \Omega \to \Rsp$ solution to \eqref{eq:ot-dual}. 
This follows from the fact that any such $\varphi$ corresponds to the $\bar{c}$-transform of the function $\varphi^c$. In turn, the Lipschitz behavior of $\xi_p$ and $\varphi$ allows to ensure their differentiability almost-everywhere using Rademacher's theorem. Consider now an optimal transport plan $\gamma$ minimizer of \eqref{eq:ot-primal} and a Kantorovich potential $\varphi$ maximizer of \eqref{eq:ot-dual}. The primal-dual relationship \eqref{eq:primal-dual} ensures that for any $(x_0, y_0) \in \spt(\gamma)$, the function $x \mapsto \xi_p(x-y_0) - \varphi(x)$ is minimized in $x_0$. Thus, almost-every $(x_0, y_0) \in \spt(\gamma)$ satisfies the optimality condition $\nabla \xi_p(x_0 - y_0) - \nabla \varphi(x_0) = 0$, which leads to
\begin{equation}
\label{eq:definition-ot-map}
    y_0 = T(x_0) := x_0 - (\nabla \xi_p)^{-1} (\nabla \varphi(x_0)).
\end{equation}
The mapping $T : \Omega \to \Omega$ is well defined almost-everywhere. These considerations show that if $\rho$ is absolutely continuous with respect to the Lebesgue measure, $\gamma$ is induced by the map $T$ defined in \eqref{eq:definition-ot-map}, i.e. $\gamma = (\id, T)_\# \rho$. Because the choices of $\gamma$ and $\varphi$ were not made depending on each other, these ideas also show the uniqueness of $\gamma$ and of $\nabla \varphi$. The map $T$ is referred to as the \emph{optimal transport map with respect to the ground cost $c = \xi_p$} in the transport between $\rho$ and $\mu$.

\subsection{Stability estimate for the pushforward operation} We now state our main result, that 
brings a tight answer to Problem~\ref{pb:statement}:

\begin{theorem}
\label{th:stability-pushforwards}
    Let $p \geq 2$ and consider the $p$-cost $c(x,y) = \xi_p(x-y) = \nr{x-y}^p$. Let $\rho, \tilde{\rho} \in \Prob(\Omega)$ where $\Omega = B(0, R)$ with $R > 0$. Assume that $\rho$ is absolutely continuous with density bounded from above by $M_\rho \in (0, +\infty)$. Let $\varphi \in \Class(\Omega)$ satisfying $\varphi = (\varphi^c)^{\bar{c}}$. Let $\tilde{\gamma} \in \Prob(\Omega \times \Omega)$ be such that $(p_1)_\# \tilde{\gamma} = \tilde{\rho}$ and assume that $\spt(\tilde{\gamma}) \subset \partial^c \varphi$. Introduce the optimal transport map $T_\varphi : \Omega \to \Omega$ which satisfies for almost-every $x \in \Omega$,
    $$ T_\varphi(x) = x - (\nabla \xi_p)^{-1} (\nabla \varphi(x)). $$
    Then, for any $q \in (p-1, \infty)$ and $r \in (1, \infty)$,
    $$ \Wass_q((T_\varphi)_\# \rho, (p_2)_\# \tilde{\gamma} ) \leq c_{d, q, p, R, M_\rho} \Wass_r(\rho, \tilde{\rho})^{\frac{r}{q(r+1)}}, $$
    where $c_{d, q, p, R, M_\rho} = 2^{8(d+1)} p^3  \left( \frac{q}{q - p + 1}\right)^{1/q} d^2(1 + \beta_d) (1+M_\rho)  (1 + R)^{2 + p + d}$, with $\beta_d$ denoting the volume of the unit ball of $\Rsp^d$. It also holds, with the same constant,
    $$ \Wass_q((T_\varphi)_\# \rho, (p_2)_\# \tilde{\gamma} ) \leq c_{d, q, p, R, M_\rho} \Wass_\infty(\rho, \tilde{\rho})^{\frac{1}{q}} .$$
\end{theorem}


\begin{remark}[Case of $\varphi \in\Class^1$] 
Whenever $\varphi$ is differentiable $\tilde{\rho}$-almost-everywhere, Theorem~\ref{th:stability-pushforwards} ensures for all $q>p-1$ and $r>1$ the following stability result for the pushforward operation by $T_\varphi$:
$$ \Wass_q((T_\varphi)_\# \rho, (T_\varphi)_\# \tilde{\rho} ) \leq c_{d, q, p, R, M_\rho} \Wass_r(\rho, \tilde{\rho})^{\frac{r}{q(r+1)}}. $$
\end{remark}

\begin{remark}[Case of $\varphi \in\Class^{1,\alpha}$] 
If the potential $\varphi$ was regular in the previous proposition, e.g. $\varphi \in \Class^{1,\alpha}(\Omega)$, one would trivially get an estimate of the form
$$ \Wass_q((T_\varphi)_\# \rho, (T_\varphi)_\# \tilde{\rho} )  \leq C\Wass_r(\rho,\tilde{\rho})^{\frac{\alpha }{p-1}},$$
relying on Lemma~\ref{lemma:pties-p-cost} and for any $q, r \geq 1$ that are such that $r \geq \frac{\alpha q }{p-1}$. 
However, as noticed in the introduction, even for $p=2$, getting regularity estimates for optimal transport potentials  requires to make strong regularity assumptions on the involved measures which are rarely satisfied in applications. When $p\neq 2$, the situation is even worse since the cost fails to satisfy the so-called Ma-Trudinger-Wang condition which, as shown in Theorem 3.1 in \cite{loeper2009regularity}, is in fact necessary for the $C^1$ regularity of optimal potentials . 
\end{remark}

\begin{remark}[Tightness of exponents] 
    The estimate of Theorem~\ref{th:stability-pushforwards} is tight in terms of exponents. 
    This follows from the following generalization of Example~\ref{ex:holder-behavior} 
    (Figure~\ref{fig:holder-behavior}). 
    In dimension $d=1$, consider on $\Omega = [-1, 1]$ the probability measures 
    $\rho = \lambda_{[-\frac{1}{2}, \frac{1}{2}]}$ and 
    $\rho^\eps = \lambda_{[-\frac{1}{2}, -\frac{\eps}{2}] \cup [\frac{\eps}{2}, \frac{1}{2}]} + \eps \delta_0$ 
    where $\eps \in (0, \frac{1}{2})$ and $\lambda_I$ denotes the  Lebesgue measure restricted 
    to a set $I$. For a given $p \geq 2$, define on $\Omega$ the potential 
    $\varphi : x \mapsto (1- \abs{x})^p$. This potential satisfies 
    $\varphi = (\varphi^c)^{\bar{c}}$ where $c$ is the $p$-cost. 
    Introduce $T_\varphi$ the associated optimal transport map, which satisfies 
    $T_\varphi(x) = \sign(x)$, and $\gamma^\eps \in \Prob(\Omega \times \Omega)$ defined with 
    $\gamma^\eps = \int_{[-\frac{1}{2}, -\frac{\eps}{2}] \cup [\frac{\eps}{2}, \frac{1}{2}]} \delta_x \otimes \delta_{T_\varphi (x)} \dd x + \eps \delta_{(0,1)}$. 
    Then $(p_1)_\# \gamma^\eps = \rho^\eps$, $\spt(\gamma^\eps) \subset \partial^c \varphi$. 
    One then has $(T_\varphi)_\# \rho = \frac{1}{2}(\delta_{-1} + \delta_{+1})$ and $(p_2)_\# \gamma^\eps = \frac{1-\eps}{2} \delta_{-1} + \frac{1+\eps}{2} \delta_{+1}$. Thus for any $q \geq 1$, $\Wass_q((T_\varphi)_\# \rho, (p_2)_\# \gamma^\eps) \sim \eps^{1/q}$, while for any $r \geq 1$ one easily has $\Wass_r(\rho, \rho^\eps) \sim \eps^{\frac{r+1}{r}}$, that is
$$ \Wass_q((T_\varphi)_\# \rho, (p_2)_\# \gamma^\eps) \sim \Wass_r(\rho, \rho^\eps)^{\frac{r}{q(r+1)}} .$$
\end{remark}

\begin{remark}[Comparison with stochastic approximations]
    The tight estimates of Theorem~\ref{th:stability-pushforwards} tend to indicate that, in dimension $d \geq 2$, the stochastic approximations of the measure $(T_\varphi)_\# \rho$ from this theorem converge more rapidly than the deterministic approximations built from $\rho$. Indeed, given a budget of $N \geq 1$ points, one can build an approximation $\tilde{\rho}_N \in \Prob(\Omega)$ of $\rho \in \Prob(\Omega)$ supported on a grid of $N$ points and that satisfies
    $$ \Wass_\infty(\rho, \tilde{\rho}_N) \lesssim N^{-1/d}.$$
    The bound in Theorem~\ref{th:stability-pushforwards} then ensures formally for any $q \geq p-1$:
    $$ \Wass_q^q( (T_\varphi)_\# \rho, (T_\varphi)_\# \tilde{\rho}_N ) \lesssim N^{-1/d}. $$
    Meanwhile, if one samples $N$ points $(x_i)_{1 \leq i \leq N}$ from $\rho$ and denotes $\hat{\rho}_N = \frac{1}{N} \sum_{i=1}^N \delta_{x_i}$ the corresponding empirical measure, Theorem 1 of \cite{Fournier2015} ensures:
    \begin{equation*}
    \Esp \Wass_q^q((T_\varphi)_\# \rho, (T_\varphi)_\# \hat{\rho}_N) \lesssim \left\{
    \begin{array}{ll}
        N^{-1/2} & \mbox{if } d < 2q, \\
        N^{-1/2} \log(1 + N) & \mbox{if } d = 2 q, \\
        N^{-q/d} & \mbox{else.}
    \end{array}
\right.
\end{equation*}
In particular, except in dimension one, the stochastic approximation $(T_\varphi)_\# \hat{\rho}_N$ converges faster (in expectation) towards $(T_\varphi)_\# \rho$ than its deterministic counterpart $(T_\varphi)_\# \tilde{\rho}_N$.
\end{remark}


The proof of Theorem~\ref{th:stability-pushforwards} relies on the following lemma, that is a direct consequence of Lemma~\ref{lemma:pties-p-cost} and whose proof is deferred after the proof of Theorem~\ref{th:stability-pushforwards}. 
\begin{lemma} \label{lemma:link-potential-convex-function}
    With the notations of Theorem~\ref{th:stability-pushforwards}, the function $\phi : x \mapsto p(p-1) R^{p-2}\frac{\nr{x}^2}{2}  - \varphi(x)$
is convex and $p^2 R^{p-1}$-Lipschitz continuous on $\Omega$. This function can be extended to a convex and $p^2 R^{p-1}$-Lipschitz continuous function defined on $\Rsp^d$. Moreover, for any $x \in \Omega$ and $\eta > 0$, 
$$ \diam( \partial^c \varphi(B(x, \eta) \cap \Omega) ) \leq 8 p (1+ R^{\frac{p-2}{p-1}}) \left( \eta^{\frac{1}{p-1}} + \diam( \partial \phi(B(x, \eta)) )^{\frac{1}{p-1}} \right).$$
\end{lemma}


We are now ready to prove Theorem~\ref{th:stability-pushforwards}.
\begin{proof}[Proof of Theorem~\ref{th:stability-pushforwards}]
In this proof, we omit for clarity the multiplicative constants that depend on $d, q, p, R$ or $M_\rho$ and use $\lesssim$ instead of $\leq$ for inequalities involving such constants. A close look at this proof allows to recover the multiplicative constant of the statement. Let us assume for now that $r \in (1, \infty) \cup \{\infty\}$. We will deal with each of the distinct cases $r = \infty$ and $r < \infty$ afterwards.

We first disintegrate $\tilde{\gamma}$ with respect to $\tilde{\rho}$, i.e. we let $\tilde{\gamma} = \int \delta_x\otimes \tilde{\gamma}_x \dd\tilde{\rho}$, where $x \mapsto \tilde{\gamma}_x$ is a measurable map from $\Omega$ to $\Prob(\Omega)$. By assumption, the support of $\tilde{\gamma}$ is included in $\partial^c\varphi$. This implies that for any $x$ in $\Omega$, the support of $\tilde{\gamma}_x$ is included in $\partial^c \varphi(x) = \{y \in \Omega \mid \varphi(x) + \varphi^c(y) = c(x,y)\}$. We introduce $S:\Rsp^d \to \Rsp^d$ an optimal transport map from $\rho$ to $\tilde{\rho}$ for the $r$-cost\footnote{For $r=\infty$, the existence of an optimal transport map was first established in \cite{champion08}.}
and we consider the measure $\gamma = \int \delta_x \otimes \tilde{\gamma}_{S(x)} \dd \rho(x)$. This measure $\gamma$ is a coupling between $\rho$ and $(p_2)_\# \tilde{\gamma}$, which implies that  $(T_\varphi, \id)_\# \gamma$ is a coupling between $(T_\varphi)_\# \rho$ and $(p_2)_\# \tilde{\gamma}$. These constructions may be summarized by the following diagram.

\begin{center}
  \begin{tikzcd}[row sep=4em, column sep=7em]
    \rho \arrow{r}{S}\arrow{d}{T_\varphi} \arrow{dr}{\gamma} & \tilde{\rho} \arrow{d}{\tilde{\gamma}} \\
     (T_\varphi)_{\#} \rho \arrow{r}{(T_\varphi,\id)_{\#}\gamma}  & p_{2\#} \tilde{\gamma}
  \end{tikzcd}
\end{center}
We therefore have the bound:
\begin{align}
\label{eq:upper-bound-1-pushforwards}
    \Wass_q^q((T_\varphi)_\# \rho, (p_2)_\# \tilde{\gamma}) &\leq \int_{\Omega \times \Omega} \nr{T_\varphi(x) - y}^q \dd \gamma(x,y) \notag \\
    &=  \int_{\Omega \times \Omega} \nr{T_\varphi(x) - y}^q \dd \tilde{\gamma}_{S(x)}(y) \dd \rho(x) \notag \\
    &= \int_{x \in \Omega} \int_{y \in \partial^c \varphi(S(x))} \nr{T_\varphi(x) - y}^q \dd \tilde{\gamma}_{S(x)}(y) \dd \rho(x),
\end{align}
where we used that $\spt(\tilde{\gamma}_{S(x)}) \subseteq \partial^c \varphi(S(x))$  to get the last line. For a given $\eta \in (0, 2 R + 1]$, we will upper bound the right-hand side 
by splitting the integral on $\Omega_\eta$ and $\Omega_\eta^c$, where 
$$ \Omega_\eta = \left\{x \in \spt(\rho) \vert \nr{S(x) - x} \leq \eta \right\}, \qquad \Omega_\eta^c = \spt(\rho) \setminus \Omega_\eta. $$


\noindent\emph{\textbf{Upper bound on $\Omega_\eta$.}}
By definition, any point $x$ in  $\Omega_\eta$ satisfies $\nr{S(x) - x} \leq \eta $, so that $S(x)$ belongs to the ball $B(x, \eta)$ intersected with $\Omega$. 
 Then for any such $x$, $\partial^c \varphi(S(x)) \subset \partial^c \varphi(B(x, \eta) \cap \Omega)$. Therefore for any $g \in \partial^c \varphi(x)$ and $y \in \partial^c \varphi(S(x))$, one has
$$ \nr{g - y} \leq \diam \left(\partial^c \varphi(B(x, \eta) \cap \Omega)\right), $$
so that, recalling that $T_\varphi(x) \in \partial^c \varphi(x)$, the quantity $$\int_{x \in \Omega_\eta} \int_{y \in \partial^c \varphi(S(x))} \nr{T_\varphi(x) - y}^q \dd \tilde{\gamma}_{S(x)}(y) \dd \rho(x)$$ is dominated by
\begin{align*}
    \int_{x \in \Omega_\eta} \diam \left(\partial^c \varphi(B(x, \eta) \cap \Omega)\right)^q \dd \rho(x).
\end{align*}
Let $\phi$ be the convex and $p^2 R^{p-1}$-Lipschitz function on $\Omega$ defined from $\varphi$ in Lemma~\ref{lemma:link-potential-convex-function}. This lemma ensures that:
$$\diam( \partial^c \varphi(B(x, \eta) \cap \Omega) ) \lesssim  \eta^{\frac{1}{p-1}} + \diam( \partial \phi(B(x, \eta)) )^{\frac{1}{p-1}}.$$
We thus have the estimate:
\begin{equation*}
    \int_{x \in \Omega_\eta} \diam \left(\partial^c \varphi(B(x, \eta) \cap \Omega)\right)^q \dd \rho(x) \lesssim \eta^{\frac{q}{p-1}} + \int_\Omega \diam( \partial \phi(B(x, \eta)) )^{\frac{q}{p-1}} \dd \rho(x).
\end{equation*}
Using that $\frac{q}{p-1} > 1$ and that $\eta \leq 2 R + 1$, one has
\begin{equation*}
    \eta^{\frac{q}{p-1}} = (2 R + 1)^{\frac{q}{p-1}} \left( \frac{\eta}{2 R + 1}\right)^{\frac{q}{p-1}} \lesssim \eta.
\end{equation*}
Using again that $\frac{q}{p-1} > 1$, Corollary~\ref{cor:estimate-grad-phi} ensures the bound:
\begin{equation*}
    \int_{\Omega} \diam(\partial \phi(B(x, \eta)))^{ \frac{q}{p-1}} \dd \rho(x) \lesssim \eta,
\end{equation*}
The last two bounds thus entail
\begin{equation*}
    \int_{x \in \Omega_\eta} \diam \left(\partial^c \varphi(B(x, \eta) \cap \Omega)\right)^q \dd \rho(x) \lesssim \eta.
\end{equation*}
We therefore have the bound:
\begin{equation}
\label{eq:second-bound-markov}
    \int_{x \in \Omega_\eta} \int_{y \in \partial^c \varphi(S(x))} \nr{T_\varphi(x) - y}^q \dd \tilde{\gamma}_{S(x)}(y) \dd \rho(x) \lesssim \eta.
\end{equation}
This last bound allows to deal with the case $r= \infty$. Indeed, assuming that $r = \infty$, we get by setting $\eta = \Wass_\infty(\rho, \tilde{\rho})$ that $\Omega_\eta = \Omega$, $\Omega_\eta^c = \emptyset$, and the previous inequality combined with \eqref{eq:upper-bound-1-pushforwards} allows to reach the conclusion that
\begin{equation*}
    \Wass_q^q((T_\varphi)_\# \rho, (p_2)_\# \tilde{\gamma}) \lesssim \Wass_\infty(\rho, \tilde{\rho}).
\end{equation*}
We now assume that $r \in (1 , + \infty)$. There remains to bound the value of the integrand in \eqref{eq:upper-bound-1-pushforwards} on the domain $\Omega_\eta^c$.

\noindent\emph{\textbf{Upper bound on $\Omega_\eta^c$.}} The optimal transport map $S$ from $\rho$ to $\tilde{\rho}$ satisfies
$$ \nr{S - \id}_{\L^r(\rho)} = \Wass_r(\rho, \tilde{\rho}). $$
Then using Markov's inequality, $\eta^r \rho(\Omega_\eta^c) \leq  \Wass_r^r(\rho, \tilde{\rho})$. The fact that $T_\varphi$ is valued in $\Omega$ then implies
\begin{equation} \label{eq:first-bound-markov}
    \begin{aligned}
        \int_{x \in \Omega_\eta^c} \int_{y \in \partial^c \varphi(S(x))} \nr{T_\varphi(x) - y}^q \dd \tilde{\gamma}_{S(x)}(y) \dd \rho(x) &\leq \int_{x \in \Omega_\eta^c} (2 R)^q \dd \rho(x) \\
        &\lesssim \frac{\Wass_r^r(\rho, \tilde{\rho})}{\eta^r}.
    \end{aligned} 
\end{equation}

\noindent\emph{\textbf{Conclusion.}} Using bounds \eqref{eq:second-bound-markov} and \eqref{eq:first-bound-markov} in \eqref{eq:upper-bound-1-pushforwards} we have for any $\eta \in (0, 2 R]$:
\begin{equation*}
    \Wass_q^q((T_\varphi)_\# \rho, (p_2)_\# \tilde{\gamma}) \lesssim \frac{\Wass_r^r(\rho, \tilde{\rho})}{\eta^r}  + \eta.
\end{equation*}
Setting $\eta = \Wass_r(\rho, \tilde{\rho})^{\frac{r}{r+1}}$ then allows us to conclude:
\begin{equation*}
    \Wass_q^q((T_\varphi)_\# \rho, (p_2)_\# \tilde{\gamma}) \lesssim \Wass_r(\rho, \tilde{\rho})^{\frac{r}{r+1}}. \qedhere
\end{equation*}

\end{proof}

We conclude this section with the proof of Lemma~\ref{lemma:link-potential-convex-function}.

\begin{proof}[Proof of Lemma~\ref{lemma:link-potential-convex-function}]
    From Lemma~\ref{lemma:pties-p-cost}, we know that the $p$-cost $\xi_p$ is a $-C_{p, R}$-concave function with $C_{p, R} = p(p-1) R^{p-2}$. Since $\varphi$ verifies $\varphi = (\varphi^c)^{\bar{c}}$, it is also a $-C_{p, R}$-concave function as an infimum of $-C_{p, R}$-concave functions. In particular, the function $\phi$ is a convex function. Similarly, Lemma~\ref{lemma:pties-p-cost} ensures that $\xi_p$ is $p R^{p-1}$-Lipschitz continuous on $\Omega$ and so is $\varphi = (\varphi^c)^{\bar{c}}$. An immediate computation then ensures that $\phi$ is $p^2 R^{p-1}$-Lipschitz continuous on $\Omega$. The $p^2 R^{p-1}$-Lipschitz continuous convex function $\phi$ defined on $\Omega$ can be extended, by mean of double convex conjugate, as a $p^2 R^{p-1}$-Lipschitz continuous convex function defined on the whole $\Rsp^d$ and coinciding with $\phi$ on $\Omega$:
\begin{equation*}
    \forall x \in \Rsp^d, \quad \phi(x) := \sup_{x_\Omega \in \Omega, g \in \partial \phi(x_\Omega)} \phi(x_\Omega) + \sca{g}{x - x_\Omega}. 
\end{equation*}
Now consider $x \in \Omega$ and $\eta > 0$. Let $x^-, x^+ \in B(x, \eta) \cap \Omega$.  Let $y^- \in \partial^c \varphi(x^-)$ and $y^+ \in \partial^c \varphi(x^+)$. We want to bound $\nr{y^+ - y^-}$ in terms of $\eta$ and $\diam( \partial \phi(B(x, \eta)) )$. Let's refer for now to $x^-, y^-$ and $x^+, y^-$ indistinctly with $x^\pm, y^\pm$. Recall that
$$ \partial^c \varphi(x^\pm) = \{ y \in \Omega \mid \varphi(x^\pm) + \varphi^c(y) = \xi_p(x^\pm - y) \}.$$
Therefore, $y^\pm \in \partial^c \varphi(x^\pm) $ if and only if $z \mapsto \xi_p(z - y^\pm) - \varphi(z)$ is minimized in $x^\pm$, that is if and only if
$$z \mapsto \xi_p(z - y^\pm) - C_{p, R} \frac{\nr{z}^2}{2} + \phi(z)$$
is minimized in $x^\pm$. This is possible only if there exists $g^\pm \in \partial \phi(x^\pm)$ such that
$$ 0 = \nabla \xi_p(x^\pm - y^\pm) - C_{p, R} x^\pm + g^\pm. $$
Hence there exists $g^\pm \in \partial \phi(x^\pm)$ such that
$$ y^\pm = x^\pm - (\nabla \xi_p)^{-1}(  C_{p, R} x^\pm - g^\pm). $$
Considering such subgradients $g^\pm \in \partial \phi(x^\pm)$, we thus have:
\begin{align*}
    \nr{y^+ - y^-} \leq \nr{x^+ - x^-} + \nr{ (\nabla \xi_p)^{-1}(  C_{p, R} x^+ - g^+) - (\nabla \xi_p)^{-1}(  C_{p, R} x^- - g^-) }.
\end{align*}
The Hölder behavior of $(\nabla \xi_p)^{-1}$ described in Lemma~\ref{lemma:pties-p-cost} then allows us to write:
\begin{align*}
     \nr{y^+ - y^-} \leq \nr{x^+ - x^-} + \frac{3}{p^{\frac{1}{p-1}}} \nr{  C_{p, R} (x^+ - x^-) - g^+  + g^- }^{\frac{1}{p-1}}.
\end{align*}
Therefore, using that $x^\pm \in B(x, \eta)$, we have that $\nr{g^+ - g^-} \leq \diam(\partial \phi(B(x, \eta)))$ so that we have the bound
\begin{equation*}
    \nr{y^+ - y^-} \leq 2 \eta + \frac{3}{p^{\frac{1}{p-1}}} (C_{p, R} \eta)^{\frac{1}{p-1}} + \frac{3}{p^{\frac{1}{p-1}}} \diam(\partial \phi(B(x, \eta)))^{\frac{1}{p-1}}.
\end{equation*}
Maximizing over $y^-$ and $y^+$ and using that $\eta \leq R$ leads to the bound:
\begin{equation*}
    \diam( \partial^c \varphi( B(x, \eta) \cap \Omega) ) \leq  8 p (1+ R^{\frac{p-2}{p-1}}) \left(\eta^{\frac{1}{p-1}} +  \diam(\partial \phi(B(x, \eta)))^{\frac{1}{p-1}} \right). \qedhere
\end{equation*}
\end{proof}

%% file: 9_appendix.tex

\subsection{Proof of Lemma~\ref{lemma:pties-p-cost}}

\begin{proof}[Proof of Lemma~\ref{lemma:pties-p-cost}]
    The strict convexity of $\xi_p$ results from the triangle inequality and the strict convexity of $u \mapsto u^p$ on $\Rsp_+^*$ for $p\geq2$.
    
    Denoting $z_i$ the $i$-th coordinate of $z \in \Omega$ in the canonical basis of $\Rsp^d$, one has $\xi_p(z) = ( \sum_{i=1}^d z_i^2 )^{p/2}$. From this expression we deduce immediately that $\xi_p$ is of class $\Class^2$ and that its gradient and hessian read respectively
    \begin{gather*}
    \nabla \xi_p(z) = p z \nr{z}^{p-2} \quad \text{and} \quad \nabla^2 \xi_p(z) = p \nr{z}^{p-2} \id + p(p-2) \nr{z}^{p-4} z z^\top.
    \end{gather*}
    Thus for all $z \in \Omega$, $\nr{\nabla \xi_p(z)} \leq p R^{p-1}$ and $\xi_p$ is $p R^{p-1}$-Lipschitz continuous.
    
    For any $v \in \Rsp^d$, one has
    \begin{equation*}
        v^\top \nabla^2 \xi_p(z) v =  p \nr{z}^{p-2} \nr{v}^2 + p(p-2) \nr{z}^{p-4} \sca{v}{z}^2.
    \end{equation*}
    Cauchy-Schwartz inequality entails $\sca{v}{z}^2 \leq \nr{v}^2\nr{z}^2$, so that
    \begin{equation*}
        v^\top \nabla^2 \xi_p(z) v \leq p(p-1)\nr{z}^{p-2} \nr{v}^2.
    \end{equation*}
    For any $z \in \Omega$ we thus have the bound 
    \begin{equation*}
        0 \preceq \nabla^2 \xi_p(z) \preceq p(p-1) R^{p-2},
    \end{equation*}
    from which we deduce that $z \mapsto \xi_p(z) - p(p-1) \frac{R^{p-2}}{2} \nr{z}^2$ is a concave function.
    
    Finally, the mapping $z \mapsto \nabla \xi_p(z) = p z \nr{z}^{p-2}$ is obviously bijective on $\Rsp^d$. For any $y, z \in \Rsp^d \setminus \{0\}$ such that $\nabla \xi_p(y) = z$, one has 
    $$ z = p y \nr{y}^{p-2},$$
    so that $\nr{z} = p \nr{y}^{p-1}$. From this fact one deduces $$y = (\nabla \xi_p)^{-1}(z) = \frac{1}{p^{1/(p-1)}} \frac{z}{\nr{z}^\frac{p-2}{p-1}} = \frac{1}{p^{\beta(p)} \nr{z}^{1 - \beta(p)} } z,$$
    where $\beta(p) = \frac{1}{p-1} \in (0, 1]$ for $p \geq 2$.
    
    Let's finally show the Hölder behavior of $(\nabla \xi_p)^{-1}$. Let $x, y \in \Rsp^d$. If $x = 0$ and $y \neq 0$, then
    \begin{equation*}
    \nr{(\nabla \xi_p)^{-1}(y) - (\nabla \xi_p)^{-1}(x)} = \nr{(\nabla \xi_p)^{-1}(y) } = \frac{1}{p^{\beta(p)} } \nr{y}^{\beta(p)},
\end{equation*}
which corresponds to a $\beta(p)$-Hölder behavior of $\xi_p$ near 0. Assume now that $x \neq 0$ and $y \neq 0$. Assume for now that $x$ and $y$ are positively linearly dependent, i.e. there exists $\lambda \geq 1$ such that $y = \lambda x$. Then:
\begin{equation*}
    \nr{(\nabla \xi_p)^{-1}(y) - (\nabla \xi_p)^{-1}(x)} = \frac{1}{p^{\beta(p)} }(\lambda^{\beta(p)} - 1) \nr{x}^{\beta(p)}.
\end{equation*}
Using that for any $u>0$ and $\beta \in (0, 1]$, $(1+u)^\beta -1 \leq u^\beta$, we have $\lambda^{\beta(p)} - 1 \leq (\lambda - 1)^{\beta(p)}$. Hence we deduce:
\begin{equation}
\label{eq:holder-bound-same-direction}
    \nr{(\nabla \xi_p)^{-1}(y) - (\nabla \xi_p)^{-1}(x)} \leq \frac{1}{p^{\beta(p)} }(\lambda - 1)^{\beta(p)} \nr{x}^{\beta(p)} = \frac{1}{p^{\beta(p)} } \nr{y - x}^{\beta(p)}.
\end{equation}
Assume now that $\nr{x} = \nr{y}$. Then:
\begin{align}
    \nr{(\nabla \xi_p)^{-1}(y) - (\nabla \xi_p)^{-1}(x)} &= \frac{1}{p^{\beta(p)} \nr{x}^{1 - \beta(p)} } \nr{y - x} \notag \\
    &= \frac{ \nr{y - x}^{1-\beta(p)} }{p^{\beta(p)} \nr{x}^{1 - \beta(p)} } \nr{y - x}^{\beta(p)} \notag \\
    &\leq \frac{ (\nr{x} + \nr{y})^{1-\beta(p)} }{p^{\beta(p)} \nr{x}^{1 - \beta(p)} } \nr{y - x}^{\beta(p)} \notag \\
    &= \frac{2^{1 - \beta(p)}}{p^{\beta(p)}} \nr{y - x}^{\beta(p)}.
    \label{eq:holder-bound-same-norm}
\end{align}
Finally, without making any assumption on $x$ and $y$, we have:
\begin{align*}
    \nr{(\nabla \xi_p)^{-1}(y) - (\nabla \xi_p)^{-1}(x)} &\leq \nr{(\nabla \xi_p)^{-1}(y) - (\nabla \xi_p)^{-1}(\frac{\nr{x}}{\nr{y}} y) } \\
    &\quad + \nr{(\nabla \xi_p)^{-1}(\frac{\nr{x}}{\nr{y}} y) - (\nabla \xi_p)^{-1}(x)}.
\end{align*}
Bound \eqref{eq:holder-bound-same-direction} ensures:
\begin{align*}
    \nr{(\nabla \xi_p)^{-1}(y) - (\nabla \xi_p)^{-1}(\frac{\nr{x}}{\nr{y}} y) } &\leq \frac{1}{p^{\beta(p)} } \nr{ y - \frac{\nr{x}}{\nr{y}} y }^{\beta(p)} \\
    &= \frac{1}{p^{\beta(p)} }  \abs{ \nr{y} - \nr{x} }^{\beta(p)} \\
    &\leq \frac{1}{p^{\beta(p)} } \nr{y-x}^{\beta(p)}.
\end{align*}
On the other hand, bound \eqref{eq:holder-bound-same-norm} ensures:
\begin{align*}
    \nr{(\nabla \xi_p)^{-1}(\frac{\nr{x}}{\nr{y}} y) - (\nabla \xi_p)^{-1}(x)} &\leq \frac{2^{1 - \beta(p)}}{p^{\beta(p)}} \nr{\frac{\nr{x}}{\nr{y}}y - x}^{\beta(p)} \\
    &\leq \frac{2^{1 - \beta(p)}}{p^{\beta(p)}} \left( \nr{\frac{\nr{x}}{\nr{y}}y - y} + \nr{y - x}\right)^{\beta(p)} \\
    &= \frac{2^{1 - \beta(p)}}{p^{\beta(p)}} \left( \abs{\nr{x} - \nr{y}} + \nr{y - x}\right)^{\beta(p)} \\ 
    &\leq \frac{2}{p^{\beta(p)}} \nr{y - x}^{\beta(p)}.
\end{align*}
We thus get eventually:
\begin{equation*}
    \nr{(\nabla \xi_p)^{-1}(y) - (\nabla \xi_p)^{-1}(x)} \leq \frac{3}{p^{\beta(p)}} \nr{y - x}^{\beta(p)}. \qedhere
\end{equation*}
\end{proof}


%% file: main.bbl
\begin{thebibliography}{10}

\bibitem{agueh:hal-00637399}
Martial Agueh and Guillaume Carlier.
\newblock {Barycenters in the Wasserstein space}.
\newblock {\em {SIAM Journal on Mathematical Analysis}}, 43(2):904--924, 2011.

\bibitem{Alberti1992}
Giovanni Alberti, Luigi Ambrosio, and Piermarco Cannarsa.
\newblock On the singularities of convex functions.
\newblock {\em manuscripta mathematica}, 76(1):421--435, Dec 1992.

\bibitem{Ambrosio1993}
Luigi Ambrosio, Piermarco Cannarsa, and Halil~Mete Soner.
\newblock On the propagation of singularities of semi-convex functions.
\newblock {\em Annali della Scuola Normale Superiore di Pisa - Classe di
  Scienze}, 20(4):597--616, 1993.

\bibitem{ambrosio2008gradient}
Luigi Ambrosio, Nicola Gigli, and Giuseppe Savar{\'e}.
\newblock {\em Gradient flows: in metric spaces and in the space of probability
  measures}.
\newblock Springer Science \& Business Media, 2008.

\bibitem{pmlr-v70-amos17b}
Brandon Amos, Lei Xu, and J.~Zico Kolter.
\newblock Input convex neural networks.
\newblock In Doina Precup and Yee~Whye Teh, editors, {\em Proceedings of the
  34th International Conference on Machine Learning}, volume~70 of {\em
  Proceedings of Machine Learning Research}, pages 146--155. PMLR, 06--11 Aug
  2017.

\bibitem{LOT4}
Saurav Basu, Soheil Kolouri, and Gustavo~K. Rohde.
\newblock Detecting and visualizing cell phenotype differences from microscopy
  images using transport-based morphometry.
\newblock {\em Proceedings of the National Academy of Sciences},
  111(9):3448--3453, 2014.

\bibitem{Brenier89}
Yann Brenier.
\newblock The least action principle and the related concept of generalized
  flows for incompressible perfect fluids.
\newblock {\em Journal of the American Mathematical Society}, 2(2):225--255,
  1989.

\bibitem{bunne22}
Charlotte Bunne, Andreas Krause, and Marco Cuturi.
\newblock Supervised training of conditional monge maps.
\newblock In S.~Koyejo, S.~Mohamed, A.~Agarwal, D.~Belgrave, K.~Cho, and A.~Oh,
  editors, {\em Advances in Neural Information Processing Systems}, volume~35,
  pages 6859--6872. Curran Associates, Inc., 2022.

\bibitem{Caffarelli1}
Luis~A. Caffarelli.
\newblock The regularity of mappings with a convex potential.
\newblock {\em Journal of the American Mathematical Society}, 5(1):99--104,
  1992.

\bibitem{Caffarelli2}
Luis~A. Caffarelli.
\newblock Boundary regularity of maps with convex potentials--ii.
\newblock {\em Annals of Mathematics}, 144(3):453--496, 1996.

\bibitem{LOT6}
Tianji Cai, Junyi Cheng, Nathaniel Craig, and Katy Craig.
\newblock Linearized optimal transport for collider events.
\newblock {\em Phys. Rev. D}, 102:116019, Dec 2020.

\bibitem{Cannarsa2021}
Piermarco Cannarsa and Wei Cheng.
\newblock Singularities of solutions of hamilton--jacobi equations.
\newblock {\em Milan Journal of Mathematics}, 89(1):187--215, Jun 2021.

\bibitem{Cannarsa87}
Piermarco Cannarsa and Halil~Mete Soner.
\newblock On the singularities of the viscosity solutions to
  hamilton–jacobi–bellman equations.
\newblock {\em Indiana University Mathematics Journal}, 36(3):501--524, 1987.

\bibitem{Cannarsa89}
Piermarco Cannarsa and Halil~Mete Soner.
\newblock Generalized one-sided estimates for solutions of hamilton-jacobi
  equations and applications.
\newblock {\em Nonlinear Analysis: Theory, Methods \& Applications},
  13(3):305--323, 1989.

\bibitem{entropic-barycenters}
Guillaume Carlier, Katharina Eichinger, and Alexey Kroshnin.
\newblock Entropic-wasserstein barycenters: Pde characterization, regularity,
  and clt.
\newblock {\em SIAM Journal on Mathematical Analysis}, 53(5):5880--5914, 2021.

\bibitem{champion08}
Thierry Champion, Luigi De~Pascale, and Petri Juutinen.
\newblock {The $L^\infty$-Wasserstein Distance: Local Solutions and Existence
  of Optimal Transport Maps}.
\newblock {\em SIAM Journal on Mathematical Analysis}, 40(1):1--20, 2008.

\bibitem{Chen18convex}
Yize Chen, Yuanyuan Shi, and Baosen Zhang.
\newblock Optimal control via neural networks: A convex approach.
\newblock {\em ArXiv, 1805.11835}, 2018.

\bibitem{language_processing_2}
Pierre Colombo, Guillaume Staerman, Pablo Piantanida, and Chlo{\'e} Clavel.
\newblock {Automatic Text Evaluation through the Lens of Wasserstein
  Barycenters}.
\newblock In {\em {EMNLP 2021}}, Punta Cana, Dominica, November 2021.

\bibitem{pmlr-v32-cuturi14}
Marco Cuturi and Arnaud Doucet.
\newblock Fast computation of wasserstein barycenters.
\newblock In Eric~P. Xing and Tony Jebara, editors, {\em Proceedings of the
  31st International Conference on Machine Learning}, volume 32(2) of {\em
  Proceedings of Machine Learning Research}, pages 685--693, Bejing, China,
  22--24 Jun 2014. PMLR.

\bibitem{deGoes15}
Fernando de~Goes, Corentin Wallez, Jin Huang, Dmitry Pavlov, and Mathieu
  Desbrun.
\newblock Power particles: An incompressible fluid solver based on power
  diagrams.
\newblock {\em ACM Trans. Graph.}, 34(4), jul 2015.

\bibitem{delalande:hal-03164147}
Alex Delalande and Quentin Mérigot.
\newblock {Quantitative Stability of Optimal Transport Maps under Variations of
  the Target Measure}.
\newblock {\em Duke Mathematical Journal (to appear)}, 2021.

\bibitem{language_processing_1}
Pierre Dognin, Igor Melnyk, Youssef Mroueh, Jarret Ross, Cicero~Dos Santos, and
  Tom Sercu.
\newblock Wasserstein barycenter model ensembling.
\newblock In {\em International Conference on Learning Representations}, 2019.

\bibitem{Fournier2015}
Nicolas Fournier and Arnaud Guillin.
\newblock On the rate of convergence in wasserstein distance of the empirical
  measure.
\newblock {\em Probability Theory and Related Fields}, 162(3):707--738, Aug
  2015.

\bibitem{gallouet:hal-01425826}
Thomas~O Gallou{\"e}t and Quentin M{\'e}rigot.
\newblock {A Lagrangian Scheme {\`a} la Brenier for the Incompressible Euler
  Equations}.
\newblock {\em {Foundations of Computational Mathematics}}, 18:835--865, 2018.

\bibitem{pmlr-v70-ho17a}
Nhat Ho, XuanLong Nguyen, Mikhail Yurochkin, Hung~Hai Bui, Viet Huynh, and Dinh
  Phung.
\newblock Multilevel clustering via {W}asserstein means.
\newblock In Doina Precup and Yee~Whye Teh, editors, {\em Proceedings of the
  34th International Conference on Machine Learning}, volume~70 of {\em
  Proceedings of Machine Learning Research}, pages 1501--1509. PMLR, 06--11 Aug
  2017.

\bibitem{Jensen87}
R.~Jensen and P.~E. Souganidis.
\newblock A regularity result for viscosity solutions of hamilton-jacobi
  equations in one space dimensions.
\newblock {\em Transactions of the American Mathematical Society},
  301(1):137--147, 1987.

\bibitem{JKO}
Richard Jordan, David Kinderlehrer, and Felix Otto.
\newblock The variational formulation of the fokker--planck equation.
\newblock {\em SIAM Journal on Mathematical Analysis}, 29(1):1--17, 1998.

\bibitem{Kitagawa2019ANA}
Jun Kitagawa, Quentin M{\'e}rigot, and Boris Thibert.
\newblock {Convergence of a Newton algorithm for semi-discrete optimal
  transport}.
\newblock {\em J. Eur. Math. Soc.}, 21(9):2603–2651, 2019.

\bibitem{LOT7}
Soheil Kolouri and Gustavo~K. Rohde.
\newblock Transport-based single frame super resolution of very low resolution
  face images.
\newblock In {\em 2015 IEEE Conference on Computer Vision and Pattern
  Recognition (CVPR)}, pages 4876--4884, 2015.

\bibitem{LOT5}
Soheil Kolouri, Akif~B. Tosun, John~A. Ozolek, and Gustavo~K. Rohde.
\newblock A continuous linear optimal transport approach for pattern analysis
  in image datasets.
\newblock {\em Pattern Recognition}, 51:453--462, 2016.

\bibitem{korotin2021wasserstein}
Alexander Korotin, Vage Egiazarian, Arip Asadulaev, Alexander Safin, and Evgeny
  Burnaev.
\newblock Wasserstein-2 generative networks.
\newblock In {\em International Conference on Learning Representations}, 2021.

\bibitem{language_processing_3}
Xin Lian, Kshitij Jain, Jakub Truszkowski, Pascal Poupart, and Yaoliang Yu.
\newblock Unsupervised multilingual alignment using wasserstein barycenter.
\newblock In Christian Bessiere, editor, {\em Proceedings of the Twenty-Ninth
  International Joint Conference on Artificial Intelligence, {IJCAI-20}}, pages
  3702--3708. International Joint Conferences on Artificial Intelligence
  Organization, 7 2020.
\newblock Main track.

\bibitem{loeper2009regularity}
Gr\'{e}goire Loeper.
\newblock On the regularity of solutions of optimal transportation problems.
\newblock {\em Acta Math.}, 202(2):241--283, 2009.

\bibitem{levy15}
{Lévy, Bruno}.
\newblock A numerical algorithm for l2 semi-discrete optimal transport in 3d.
\newblock {\em ESAIM: M2AN}, 49(6):1693--1715, 2015.

\bibitem{pmlr-v119-makkuva20a}
Ashok Makkuva, Amirhossein Taghvaei, Sewoong Oh, and Jason Lee.
\newblock Optimal transport mapping via input convex neural networks.
\newblock In Hal~Daumé III and Aarti Singh, editors, {\em Proceedings of the
  37th International Conference on Machine Learning}, volume 119 of {\em
  Proceedings of Machine Learning Research}, pages 6672--6681. PMLR, 13--18 Jul
  2020.

\bibitem{pmlr-v108-merigot20a}
Quentin M\'erigot, Alex Delalande, and Frederic Chazal.
\newblock Quantitative stability of optimal transport maps and linearization of
  the 2-{W}asserstein space.
\newblock In {\em Proceedings of the Twenty Third International Conference on
  Artificial Intelligence and Statistics}, volume 108, pages 3186--3196, 26--28
  Aug 2020.

\bibitem{meyron18}
Jocelyn Meyron, Quentin M\'{e}rigot, and Boris Thibert.
\newblock Light in power: A general and parameter-free algorithm for caustic
  design.
\newblock {\em ACM Trans. Graph.}, 37(6), dec 2018.

\bibitem{Monge}
Gaspard Monge.
\newblock {Mémoire sur la théorie des déblais et des remblais}.
\newblock {\em {Histoire de l’Académie Royale des Sciences de Paris, avec
  les Mémoires de Mathématique et de Physique pour la même année}}, pages
  666--704, 1781.

\bibitem{MERIGOT2021133}
Quentin Mérigot and Boris Thibert.
\newblock {Chapter 2 - Optimal transport: discretization and algorithms}.
\newblock In Andrea Bonito and Ricardo~H. Nochetto, editors, {\em Geometric
  Partial Differential Equations - Part II}, volume~22 of {\em Handbook of
  Numerical Analysis}, pages 133--212. Elsevier, 2021.

\bibitem{Nakane91}
Shizuo Nakane.
\newblock {Formation of singularities for Hamilton-Jacobi equation with several
  space variables}.
\newblock {\em Journal of the Mathematical Society of Japan}, 43(1):89 -- 100,
  1991.

\bibitem{levy22}
Farnik Nikakhtar, Ravi~K. Sheth, Bruno L\'evy, and Roya Mohayaee.
\newblock Optimal transport reconstruction of baryon acoustic oscillations.
\newblock {\em Phys. Rev. Lett.}, 129:251101, Dec 2022.

\bibitem{Otto1998}
Felix Otto.
\newblock {Dynamics of Labyrinthine Pattern Formation in Magnetic Fluids: A
  Mean‐Field Theory}.
\newblock {\em Archive for Rational Mechanics and Analysis}, 141(1):63--103,
  Mar 1998.

\bibitem{otto2001geometry}
Felix Otto.
\newblock {The geometry of dissipative evolution equations: the porous medium
  equation}.
\newblock {\em Communications in Partial Differential Equations}, 26:101--174,
  2001.

\bibitem{LOT3}
S.~{Park} and M.~{Thorpe}.
\newblock Representing and learning high dimensional data with the optimal
  transport map from a probabilistic viewpoint.
\newblock In {\em 2018 IEEE/CVF Conference on Computer Vision and Pattern
  Recognition}, pages 7864--7872, 2018.

\bibitem{comp_OT}
Gabriel Peyré and Marco Cuturi.
\newblock Computational optimal transport.
\newblock {\em Foundations and Trends in Machine Learning}, 11(5-6):355--607,
  2019.

\bibitem{image_processing}
Julien Rabin, Gabriel Peyr{\'e}, Julie Delon, and Bernot Marc.
\newblock {Wasserstein Barycenter and its Application to Texture Mixing}.
\newblock In {\em {SSVM'11}}, pages 435--446, Israel, 2011. {Springer}.

\bibitem{santambrogio2015optimal}
Filippo Santambrogio.
\newblock Optimal transport for applied mathematicians.
\newblock {\em Birk{\"a}user, NY}, 55:58--63, 2015.

\bibitem{geometry_processing}
Justin Solomon, Fernando de~Goes, Gabriel Peyr\'{e}, Marco Cuturi, Adrian
  Butscher, Andy Nguyen, Tao Du, and Leonidas Guibas.
\newblock Convolutional wasserstein distances: Efficient optimal transportation
  on geometric domains.
\newblock {\em ACM Trans. Graph.}, 34(4), jul 2015.

\bibitem{JMLR:v19:17-084}
Sanvesh Srivastava, Cheng Li, and David~B. Dunson.
\newblock Scalable bayes via barycenter in wasserstein space.
\newblock {\em Journal of Machine Learning Research}, 19(8):1--35, 2018.

\bibitem{taghvaei20192wasserstein}
Amirhossein Taghvaei and Amin Jalali.
\newblock 2-wasserstein approximation via restricted convex potentials with
  application to improved training for gans, 2019.

\bibitem{Tsuji83}
Mikio Tsuji.
\newblock {Formation of singularities for Hamilton-Jacobi equation, I}.
\newblock {\em Proceedings of the Japan Academy, Series A, Mathematical
  Sciences}, 59(2):55 -- 58, 1983.

\bibitem{Tsuji86}
Mikio Tsuji.
\newblock {Formation of singularities for Hamilton-Jacobi equation II}.
\newblock {\em Journal of Mathematics of Kyoto University}, 26(2):299 -- 308,
  1986.

\bibitem{villani2008optimal}
C{\'e}dric Villani.
\newblock {\em Optimal transport: old and new}, volume 338.
\newblock Springer Science \& Business Media, 2008.

\bibitem{LOT_ref_image}
Wei Wang, Dejan Slep\v{c}ev, Saurav Basu, John~A. Ozolek, and Gustavo~K. Rohde.
\newblock A linear optimal transportation framework for quantifying and
  visualizing variations in sets of images.
\newblock {\em Int. J. Comput. Vision}, 101(2):254--269, January 2013.

\end{thebibliography}
